\documentclass[review,onefignum,onetabnum]{siamart171218}

\usepackage{lipsum}
\usepackage{amsfonts}
\usepackage{graphicx}
\usepackage{epstopdf}
\usepackage{algorithmic}
\ifpdf
  \DeclareGraphicsExtensions{.eps,.pdf,.png,.jpg}
\else
  \DeclareGraphicsExtensions{.eps}
\fi

\newsiamremark{remark}{Remark}
\newsiamremark{hypothesis}{Hypothesis}
\crefname{hypothesis}{Hypothesis}{Hypotheses}
\newsiamthm{claim}{Claim}
\headers{Fractional Diffusion-Advection-Reaction}{V. Ginting and Y. Li}

\title{On Fractional Diffusion-Advection-Reaction Equation in $\mathbb{R}$\thanks{Submitted to the editors May, 2018.
\funding{Y. Li was partially supported by the UW Science Initiative Scholarship.}}}
\author{V. Ginting\thanks{Department of Mathematics \& Statistics, University of Wyoming, Laramie, WY 
  (\email{vginting@uwyo.edu}, \email{yli25@uwyo.edu}).}
\and Y. Li\footnotemark[2]}

\usepackage{amsopn}

\usepackage{hyperref}
\usepackage{comment}
\usepackage{color}
\usepackage{mathrsfs}
\usepackage{graphicx}
\usepackage{amsmath}
\usepackage{amsfonts}
\usepackage{amssymb}
\usepackage{graphicx,epsfig}
\usepackage{bbm}
\usepackage[left=1.8cm,right=1.8cm]{geometry}

\newtheorem{property}{Property}[section]

\newcommand{\RN}[1]{%
  \textup{\uppercase\expandafter{\romannumeral#1}}%
  }

\ifpdf
\hypersetup{
  pdftitle={An Example Article},
  pdfauthor={V. Ginting and Y. Li}
}
\fi

\begin{document}
\begin{nolinenumbers}

\maketitle

\begin{abstract}
We present an analysis of existence, uniqueness, and smoothness of the solution to a class of fractional ordinary differential equations posed on the whole real line that models a steady state behavior of a certain anomalous diffusion, advection, and reaction. The anomalous diffusion is modeled by the fractional Riemann-Liouville differential operators. The strong solution of the equation is sought in a Sobolev space defined by means of Fourier Transform. The key component of the analysis hinges on a characterization of this Sobolev space with the Riemann-Liouville derivatives that are understood in a weak sense. The existence, uniqueness, and smoothness of the solution is demonstrated with the assistance of several tools from functional and harmonic analyses.
\end{abstract}



\section{Introduction}\label{sec:introduction}
Fractional integral and differential operators and fractional differential equations have gained increasingly crucial role as useful tools for modeling various anomalous and nonlocal phenomena. By no means exhaustive, some of the applications include conservation of fluid in a porous medium \cite{wmawr08}, anomalous diffusion \cite{METZLER20001}, atmospheric advection-dispersion of pollutants \cite{GOULART20179}, continuum mechanics \cite{Mainardi1997}, and dynamics in financial markets \cite{SCALAS2000376}. 

Recent years have seen very active investigations on theoretical and numerical analysis of fractional differential equations. The existence of solutions to many types of fractional differential equations have been widely studied by using functional analytic approaches with some aiming at finding analytical/closed form solutions of the problems (see, e.g. \cite{MR2218073}, \cite{MR1347689}, \cite{MR3586266}). Using unctional analytic framework and variational formulations, several numerical schemes for approximating boundary value problems involving fractional differential equations were derived and analyzed (see e.g. \cite{ernmpde06}, \cite{wyzsinum14}, and \cite{JinLPR15}). Moreover, there has been a renewed interest on investigation of fractional Sturm-Liouville boundary value problems on unbounded domains \cite{hmmjcp15}.

Among the recurring themes in the aforementioned works is on the wellposedness of the problems under investigation. When posed on a bounded domain, typically a fractional differential equation must be provided with a set of boundary conditions. However, fractional integral and differential operators are inherently nonlocal, and in this regard, the choice of suitable and correct boundary settings to accompany the equation is not immediately clear. Other related topic is on the stability and regularity of the solution, namely, questions about the smoothness of the solution and how it depends on the data. A variety of issues on the wellposedness of the problems and solutions regularity was for example addressed in \cite{doi:10.1137/120892295, BAEUMER2018408, FDBD15}.

The subject of this paper is on the existence, uniqueness, and regularity of stationary fractional ordinary differential equation modeling a certain anomalous diffusion, advection, and reaction on the whole real line, in which the anomalous diffusion is modeled by the fractional Riemann-Liouville derivatives. One can associate this equation as a study of steady state behavior of a time dependent problem containing spatial fractional derivatives (see e.g. \cite{IZSAK201738} and \cite{BENSON2013479}). In giving a proper response, there are several inquiries to address, among which are: 1) What is a suitable functional space inside of which the solution of the equation is to be sought? 2) What should be a good setting to analyze the existence, smoothness, stability of solution?

The central thesis of the current investigation is that a class of fractional Sobolev spaces is a suitable "sandbox" to search for the solutions of the said fractional ordinary differential equations. In particular, we heavily utilize the Sobolev space that is defined by means of Fourier Transform. One of the main results is an ability to relate functions in this Sobolev space to functions whose Riemann-Liouville derivatives are understood in a weak sense. In fact, we show that the Sobolev space is equal to space of functions whose Riemann-Liouville derivatives are square integrable. Once this is in place, several tools from functional and harmonic analyses are employed to certify the existence and uniqueness of the strong solution of the equation. Furthermore, under an assumption of increasing smoothness of the data, the smoothness of the solution may be revamped as well.

The rest of the paper is organized as follows. An introduction to fractional Riemann-Liouville integral and differential operators and some of their relevant properties are presented in \Cref{sec:FRLO}. After listing several well-established results on Sobolev spaces of real-valued functions in $\mathbb{R}$, discussion in \Cref{sec:CFSS} is concentrated on a characterization of $\widehat{H}^{s}(\mathbb{R})$, a Sobolev space that is defined using Fourier Transform. It is achieved through the notion of weak fractional Riemann-Liouville derivatives, whose corresponding functional spaces are shown to be identical to $\widehat{H}^{s}(\mathbb{R})$. An application of the preceding framework to demonstrate existence and uniqueness of a strong solution to a fractional diffusion-advection-reaction in $\mathbb{R}$ (see \cref{eq:FODE}) is presented in \Cref{sec:SFDARE}. The analysis in this section includes the stability and regularity estimates of the solution. Conclusion and future works is presented in \Cref{sec:concl}. A list of frequently invoked theorems is given in \Cref{app:thm}.

Several notations, conventions, definitions, and related facts to be used throughout the paper are collected in this paragraph. We assume all the functions are real valued unless otherwise specified. For a given set $\Omega \subset \mathbb{R}$, we use characteristic function $\mathbbm{1}_{\Omega}\, w(x)$ to denote $\mathbbm{1}_{\Omega}\, w(x)=
\begin{cases}
w(x), \, x\in \Omega,\\
0,\, x\in \mathbb{R}\backslash \Omega,
\end{cases}
$
for any function $w$ defined in $\Omega$ (even though $w$ may not be defined on $\mathbb{R}\backslash \Omega$).
Let
\[
\| w \|_{L^p(\Omega)} :=
\begin{cases}
\displaystyle \left( \int_\Omega | w(x) |^p \, {\rm d} x \right)^{1/p}, ~~ \text{for}~ 1 \le p < \infty, \\
\displaystyle  \text{ess sup} \{ |w(x)| : x \in \Omega \}, ~~ \text{for} ~ p = \infty.
 \end{cases}
\]
The Lebesgue spaces $L^p(\Omega)$ is defined as $\{ w: \Omega \rightarrow \mathbb{R} :  \| w \|_{L^p(\Omega)}<\infty \}$. We note that $L^2(\Omega)$ is a Hilbert space and $(\cdot,\cdot)_{L^2(\Omega)}$ denotes its usual inner product that generates its norm $\| \cdot \|_{L^2(\Omega)}$. To simplify presentation, we use $(\cdot, \cdot)$ when $\Omega = \mathbb{R}$. $C_0^\infty(\mathbb{R})$ denotes the space of all infinitely differentiable functions with compact support in $\mathbb{R}$. $\mathbb{N}_0$ denotes the set of all non-negative integers. Convolution of two functions $v$ and $w$ is defined as $[v*w](t) = \int_{\mathbb{R}} v(x) w(t-x) \, {\rm d} x=\int_{\mathbb{R}} v(t-x) w(x) \, {\rm d} x$. Given $w:\mathbb{R} \rightarrow \mathbb{R}$, $[\mathcal{F}(w)](\xi) =\int_{\mathbb{R}} e^{-2\pi i x\xi}w(x) \, {\rm d}x$, for $\xi \in \mathbb{R}$, denotes the Fourier Transform of $w$. The notation $\widehat{w}$ denotes the Plancherel Transform of $w$ defined in \Cref{thm:PAR}, which coincides with $\mathcal{F}(w)$ if $w\in L^1(\mathbb{R})\cap L^2(\mathbb{R})$. The notation $w^\vee$ denotes  the inverse of Plancherel Transform. Given $h\in \mathbb{R}$, define the translation operator $\tau_h$ as $\tau_h w(x) = w(x-h)$. Also, given $\kappa>0$, define the dilation operator $\Pi_\kappa$ as $\Pi_\kappa w(x) = w(\kappa x)$. By appropriate change of variable,
$[\mathcal{F}(\tau_h w)](\xi) = e^{-2\pi i h\xi}[\mathcal{F}(w)](\xi)$, $[\mathcal{F}(\Pi_{-1} w)](\xi) = \overline{[\mathcal{F}(w)](\xi)}$, and $[\mathcal{F}(\Pi_\kappa w)](\xi) = \kappa^{-1}[\mathcal{F}(w)](\kappa^{-1}\xi)$ for $0 \ne \kappa \in \mathbb{R}$. Here $\overline{z}$ is the usual complex conjugate of $z \in \mathbb{C}$.

\section{Fractional Riemann-Liouville Operators}\label{sec:FRLO}
Definitions and several well-established facts about Riemann-Liouville (in short R-L) integrals and derivatives are laid out in this section, most of them without providing rigorous proofs. They have been recorded in various literatures, for which interested readers may refer to the specific references cited in the statements of the results.
%
\subsection{Fractional Riemann-Liouville Integrals and Their Properties}
\label{ssec:FRLI}
\begin{definition} \label{def:RLI}
Let $w:(a,b) \rightarrow \mathbb{R}, (a,b) \subset \mathbb{R}$ and  $\sigma >0$. The left and right Riemann-Liouville fractional integrals of order $\sigma$ are, formally respectively, defined as
\begin{align}
{_a}D_{x}^{-\sigma} w(x) &:= \dfrac{1}{\Gamma(\sigma)}\int_{a}^{x}(x-s)^{\sigma -1}w(s) \, {\rm d}s, \label{eq:LRLI}\\
{_x}D_{b}^{-\sigma} w(x) &:= \dfrac{1}{\Gamma(\sigma)}\int_{x}^{b}(s-x)^{\sigma-1}w(s) \, {\rm d}s,  \label{eq:RRLI}
\end{align}
where $\Gamma(\sigma)$ is the usual Gamma function. For convenience, we set
\begin{equation}\label{eq:LRRLI}
\boldsymbol{D}^{-\sigma}w(x) :={_{-\infty}}D_{x}^{-\sigma} w(x) \text{ and }
\boldsymbol{D}^{-\sigma * }w(x) :={_{x}}D_{\infty}^{-\sigma} w(x). 
\end{equation}
\end{definition}
Various aspects of these operators have been investigated in \cite{MR1347689}.
\begin{remark}\label{rem:ConvolutionofIntegral}
Each $\boldsymbol{D}^{-\sigma}w$ and $\boldsymbol{D}^{-\sigma*}w$ can be expressed as a convolution (\cite{MR1347689}, p. 94), namely,
\[
\boldsymbol{D}^{-\sigma}w= \frac{1}{\Gamma(\sigma)} f_1*w,\quad \boldsymbol{D}^{-\sigma*}w=\frac{1}{\Gamma(\sigma)} f_2*w,
\]
where
\[
f_1= \mathbbm{1}_{(0,\infty)} t^{\sigma-1}
\text{ and }
f_2 = \mathbbm{1}_{(-\infty,0)}  |t|^{\sigma-1},
\]
In particular, if $0<\sigma<1$, $f_1, f_2$ can be identified as distributions since they are locally integrable (see, e.g. \cite{MR1157815}, p. 157).
\end{remark}
\begin{property} \label{cor:SG}
Let $\mu,\sigma>0$,  $w\in L^p(\mathbb{R})$ with $1\leq p\leq \infty$. For any fixed $a, b\in \mathbb{R}$, the following is true 
\begin{equation}
\begin{aligned}
{_aD_x^{-\mu}}{_aD_x^{-\sigma}}w(x) &= {_aD_x^{-\mu - \sigma}}w(x),  \quad x> a \\
{_xD_b^{-\mu}}{_xD_b^{-\sigma}}w(x)& = {_xD_b^{-\mu - \sigma}}w(x), \quad b< x.
\end{aligned}
\end{equation}
\end{property}
\begin{proof}
For a bounded interval $(a,b)$ and $w \in L^p(a,b)$ it has been shown in 
\cite{MR2218073} Lemma 2.3, p. 73 that
\begin{align} \label{eq:ZZZ}
{_aD_x^{-\mu}}{_aD_x^{-\sigma}}w(x) = {_aD_x^{-\mu - \sigma}}w(x)  \text{ and }~
{_xD_b^{-\mu}}{_xD_b^{-\sigma}}w(x) = {_xD_b^{-\mu - \sigma}}w(x).
\end{align}
The proof below is shown only for the first equality,  the second one follows similarly.
For any $x> a$, we could always pick a integer $n$, such that $x\in (a,a+n)$. Notice $w\in L^p(a,a+n)$, applying \cref{eq:ZZZ}, we have
 \[
 {_aD_x^{-\mu}}{_aD_x^{-\sigma}}w(x) = {_aD_x^{-\mu - \sigma}}w(x), ~~x\in (a,a+n).
 \]
Since $n$ is arbitrary, this means ${_aD_x^{-\mu}}{_aD_x^{-\sigma}}w(x) = {_aD_x^{-\mu - \sigma}}w(x)$, for $x> a$.
\end{proof}
The following is an immediate consequence of \Cref{cor:SG}.
\begin{corollary}\label{cor:SGInfinity}
Let $\mu,\sigma >0$, and $w\in C_0^\infty(\mathbb{R})$, then
\begin{equation}
\boldsymbol{D}^{-\mu} \boldsymbol{D}^{-\sigma} w=\boldsymbol{D}^{-(\mu+\sigma)}w  \quad \text{and} \quad \boldsymbol{D}^{-\mu*} \boldsymbol{D}^{-\sigma*}w = \boldsymbol{D}^{-(\mu+\sigma)*}w.
\end{equation}
\end{corollary}
%
%
\begin{corollary}\label{cor:ADJInfinity}
Let $v,w\in C^\infty(\mathbb{R})$,  $\text{supp}(v)\subset(a,+\infty), \text{supp}(w)\subset(-\infty, b), \mu>0,b>a$. It is true that
\begin{equation}
(\boldsymbol{D}^{-\mu}v,w)=(v,\boldsymbol{D}^{-\mu*} w).
\end{equation}
\end{corollary}
\begin{proof}
For a bounded interval $(a,b)$ and $v,w \in L^2(a,b)$ and $\sigma>0$, it has been shown in the corollary of Theorem 3.5, p. 67 of \cite{MR1347689}, that
\begin{equation} \label{eq:ADJJ}
 (_aD_x^{-\sigma}v,w)_{L^2(a,b)} = (v,{_xD_b^{-\sigma}}w)_{L^2(a,b)}.
\end{equation}
Notice $\text{supp}(\boldsymbol{D}^{-\mu}v)\subset (a,+\infty)$ and  $ \text{supp}(\boldsymbol{D}^{-\mu*}w)\subset (-\infty, b)$. By \Cref{def:RLI} and \cref{eq:ADJJ},
\begin{equation}
(\boldsymbol{D}^{-\mu}v,w) = (_aD_x^{-\mu}v,w)_{L^2(a,b)} = (v,{_xD_b^{-\mu}}w)_{L^2(a,b)} =(v,\boldsymbol{D}^{-\mu*} w).
\end{equation}
\end{proof}
%
%
%
\begin{property}[Fourier Transform of R-L Integrals,  \cite{MR1347689}, Theorem 7.1, p.138] \label{prop:FTFI} 
Under the assumption that $w \in L^1(\mathbb{R})$ and $0<\sigma<1$,
 \begin{equation}
  \mathcal{F}(\boldsymbol{D}^{-\sigma} w ) = (2\pi i\xi)^{-\sigma} \mathcal{F}(w) \text{ and }
  \mathcal{F}(\boldsymbol{D}^{-\sigma*}w )= (-2\pi i\xi)^{-\sigma} \mathcal{F}(w),\quad \xi \ne 0.
 \end{equation}
 \end{property}
 This property is equivalently given in \cite{MR1347689} Theorem 7.1, p. 138  with a different version of the definition of Fourier Transform up to a sign $-2\pi$ in exponentiation.

\begin{remark}\label{rem:ComplexPowerFunctions}
$(\mp i\xi)^{\sigma}$ is understood as equal to $|\xi|^\sigma e^{\mp  \sigma \pi i \cdot \text{sign} (\xi)/2}$.
\end{remark}
The following property is on the commutativity of R-L integrals with translation and dilation operators.
\begin{property}[\cite{MR1347689}, pp. 95, 96]\label{pro:Translation}
Under the assumption that $\boldsymbol{D}^{-\mu}w$ and $\boldsymbol{D}^{-\mu*}w$ are well-defined, the following is true:
\begin{equation}
\begin{aligned}
\tau_h(\boldsymbol{D}^{-\mu}w)=\boldsymbol{D}^{-\mu}(\tau_hw)&,\quad \tau_h(\boldsymbol{D}^{-\mu*}w)=\boldsymbol{D}^{-\mu*}(\tau_hw)\\
\Pi_\kappa(\boldsymbol{D}^{-\mu}w)=\kappa^\mu \boldsymbol{D}^{-\mu}(\Pi_\kappa w)&,\quad
\Pi_\kappa(\boldsymbol{D}^{-\mu*}w)=\kappa^\mu \boldsymbol{D}^{-\mu*}(\Pi_\kappa w).
\end{aligned}
\end{equation}
\end{property}
\subsection{Fractional Riemann-Liouville Derivatives and Their Properties}
\begin{definition}\label{def:RLD}
Let $(a,b) \subset \mathbb{R}$ and $w:(a,b) \rightarrow \mathbb{R}$. Assume $\mu >0$ and $n$ is the smallest integer greater than $\mu$ (i.e., $n-1 \leq \mu < n$). The left and right Riemann-Liouville fractional derivatives of order $\mu$ are, formally respectively, defined as
\begin{align}
{_a}D_x^{\mu}w := \frac{{\rm d}^n}{{\rm d}x^n} \left( {_a}D_x^{\mu-n} w(x) \right),  \text{ and }~
{_x}D_b^{\mu}w := (-1)^n \frac{{\rm d}^n}{{\rm d} x^n} \left( {_x}D_b^{\mu-n} w(x) \right).
\end{align}
For convenience of notation, we set
\begin{equation}
\boldsymbol{D}^{\mu} w = {_{-\infty}}D_{x}^{\mu}w \text{ and }
\boldsymbol{D}^{\mu*}w =  {_{x}}D^{\mu}_{\infty}w .\label{def:Infty}
\end{equation}
\end{definition}
\begin{property}[\cite{MR2218073}, Lemma 2.4, p. 74 ] \label{lem:IP1}
For any $\mu >0$ and $w \in L^p(a,b)$, where $(a,b) \subset \mathbb{R}$ is a bounded interval and $1\leq p\leq \infty$, then
\begin{equation}
\begin{aligned}
{_a}D_x^{\mu} {_aD_x^{-\mu}} w  = w(x) ~\text{ and }~
{_x}D_b^{\mu} {_xD_b^{-\mu}}w = w(x).
\end{aligned}
\end{equation}
\end{property}
Two immediate consequences of \Cref{lem:IP1} are stated below.
\begin{corollary}\label{cor:InverseProperty1}
Let $\mu >0$ and $w \in L^p(\mathbb{R})$ with $1\leq p\leq \infty$. For any fixed $a,b\in \mathbb{R}$,
\begin{equation}
_aD_x^{\mu} {_aD_x^{-\mu}} w = w(x), ~~\text{for }x> a, ~\text{ and }~
_xD_b^{\mu} {_xD_b^{-\mu}}w = w(x), ~~\text{for }x< b.
\end{equation}
\end{corollary}
\begin{property} \label{lem:IP2}
Let $\mu >0$ and $(a,b) \subset \mathbb{R}$ be a bounded interval. If  $w={_aD_x^{-\mu}\psi}$ for some $\psi \in L^p(a,b)$ with $  1\leq p\leq \infty$,  then
\begin{equation} \label{eq:IP2o}
_aD_x^{-\mu}{_aD_x^{\mu}}w=w(x), \forall x\in (a,b).
\end{equation}
Furthermore, if $w\in C^\infty(a, +\infty)$ and $\text{supp}(w)\subset (a, +\infty)$, then
\begin{equation} 
_aD_x^{-\mu}{_aD_x^{\mu}}w=w(x), \forall x\in (a,+\infty).
\end{equation}
Similarly, if $w={_xD_b^{-\mu} \psi}$ for some $\psi \in L^p(a,b)$ with $1\leq p\leq \infty$ , then
\begin{equation} \label{eq:IP2tw}
_xD_b^{-\mu}{_xD_b^\mu}w=w(x), \forall x\in (a,b).
\end{equation}
And furthermore, if $w\in C^\infty(-\infty, b)$ and $\text{supp}(w)\subset (-\infty, b)$, then
\begin{equation} 
_xD_b^{-\mu}{_xD_b^\mu}w=w(x), \forall  x\in (-\infty, b).
\end{equation}
\end{property}
This property is equivalently stated by \cite{{MR1347689}} (c.f. Theorem~2.3, p.~43 combined with Theorem~2.4, p.~44). As an immediate corollary, we have:
\begin{corollary}\label{cor:IPC}
Let $0< \mu$, and $w\in C_0^\infty(\mathbb{R})$, then
\begin{equation}
\boldsymbol{D}^{-\mu} \boldsymbol{D}^\mu w=w  \quad \text{and} \quad \boldsymbol{D}^{-\mu*} \boldsymbol{D}^{\mu*}w = w.
\end{equation}
\end{corollary}
%
%
\begin{property} \label{prop:Boundedness}
Let $0<\mu $ and  $w\in C_0^\infty(\mathbb{R})$, then $\boldsymbol{D}^\mu w, \boldsymbol{D}^{\mu*} w \in L^p(\mathbb{R})$ for any $1\leq p<\infty$.
\end{property}
\begin{proof}
The proof is shown only for $\boldsymbol{D}^\mu w$, the other one can be established in a similar fashion. Since $w\in C_0^\infty(\mathbb{R})$, there exists a bounded interval $(a,b-1)$, such that $\text{supp}(w) \subset (a,b-1)$. When $\mu$ is a positive integer, then $\boldsymbol{D}^\mu w=w^{(\mu)}\in L^p(\mathbb{R})$ for any $1\leq p<\infty$. Otherwise, we can always choose a non-negative integer $n$ such that $n-1<\mu< n$. Since $w \in C^\infty_0(\mathbb{R})$, by \Cref{cor:IPC}, $w(x)=\boldsymbol{D}^{-n}v(x)$, where $v(x)=w^{(n)}(x)$ also belonging to
$C_0^\infty(\mathbb{R})$. Thus, $\boldsymbol{D}^\mu w=\boldsymbol{D}^\mu\boldsymbol{D}^{-n}v(x)$.
Applying  \Cref{cor:SGInfinity}, we know $\boldsymbol{D}^{-n}v(x)=\boldsymbol{D}^{-\mu}\boldsymbol{D}^{-(n-\mu)}v(x)$. Plugging in back gives  
\begin{equation}\label{equ:Property2.5}
\boldsymbol{D}^\mu w =\boldsymbol{D}^\mu(\boldsymbol{D}^{-\mu}\boldsymbol{D}^
{-(n-\mu)}v(x))
=(\boldsymbol{D}^\mu \boldsymbol{D}^{-\mu})\boldsymbol{D}^
{-(n-\mu)}v(x).
\end{equation}
Since $(\boldsymbol{D}^\mu \boldsymbol{D}^{-\mu})\boldsymbol{D}^{-(n-\mu)}v(x)=\mathbbm{1}_{\{a<x\}}\,({_aD_x^\mu}{_aD_x^{-\mu}}){_aD_x^{-(n-\mu)}}v(x)$, applying \Cref{cor:InverseProperty1} and plugging  back into \cref{equ:Property2.5} yields
\begin{equation}\label{equ:AAA}
\boldsymbol{D}^\mu w=\boldsymbol{D}^{-(n-\mu)}v(x).
\end{equation}
Now we consider decomposition $\boldsymbol{D}^{\mu} w = f(x)+g(x)$, where

\begin{equation}
f(x)=\mathbbm{1}_{\{ x\le b\}} \, \, \boldsymbol{D}^\mu w
\quad \text{and} \quad
g(x)=\mathbbm{1}_{\{ x> b\}} \, \, \boldsymbol{D}^\mu w.
\end{equation}
In order to show $\boldsymbol{D}^{\mu} w \in L^p(\mathbb{R})$, we only need to show $f,g\in L^p(\mathbb{R})$. First we claim $f\in L^p(\mathbb{R})$. By \cref{equ:AAA} and definition in \cref{def:Infty}, we know
\begin{equation}
f(x)= \mathbbm{1}_{\{ x\leq b\}} \, \, \boldsymbol{D}^{-(n-\mu)} v(x)
= \mathbbm{1}_{\{a<x\leq b\}} \, \, {_aD_x}^{-(n-\mu)} v(x),
\end{equation}
and thus (see for example \cite{MR1347689}, p. 48)
\[
\|f\|_{L^p(\mathbb{R})}=\|{_aD_x}^{-(n-\mu)} v\|_{L^p(a,b)} \le \dfrac{(b-a)^{n-\mu}}{(n-\mu) \Gamma(n-\mu) }\|v\|_{L^p(a,b)}<\infty. 
\]
Thus, $f\in L^p(\mathbb{R})$.
Next it is demonstrated that $g\in L^p(\mathbb{R})$. By setting $\sigma=n-\mu$ and using \Cref{def:RLD}, $g(x) = \mathbbm{1}_{\{ x>b\}} (\Gamma(\sigma))^{-1}  \RN{1}$, with
\[
\RN{1} = \dfrac{{\rm d}^n}{{\rm d}x^n}\int_{-\infty}^x(x-s)^{\sigma-1}w(s)\,{\rm d}s
= \dfrac{{\rm d}^n}{{\rm d}x^n}\int_{-\infty}^{b-1}(x-s)^{\sigma-1}w(s)\,{\rm d}s,
\]
Notice that when $x>b$,
\[
\left |\dfrac{{\rm d}^n}{{\rm d}x^n}\left((x-s)^{\sigma-1}w(s) \right) \right|\leq |w(s)| ~~\text{and} ~~ |w(s)| ~\text{is integrable over}~ (b, \infty).
\]
Therefore, application Dominated Convergence Theorem gives
\[
\RN{1}= C \int_{-\infty}^{b-1}(x-s)^{\widetilde{\sigma}}w(s)\,{\rm d}s, ~~\text{where}~~\widetilde{\sigma} = \sigma-1-n,
\]
and $C=(\sigma-1)\cdot(\sigma-2)\cdots(\sigma-n)$. 
Applying H$\ddot{\mathbf{o}}$lder's Inequality to $\RN{1}$, for $x>b$, results in
\begin{equation*}
\RN{1} \le |\RN{1}|
\le  C \|(x-s)^{\widetilde{\sigma}}\|_{L^\infty(-\infty,b-1)} \|w\|_{L^1(-\infty,b-1)}
\le C (x-b+1)^{\widetilde{\sigma}} \|w\|_{L^1(-\infty,b-1)}.
\end{equation*}
Notice that $\|w\|_{L^1(-\infty,b-1)}=\|w\|_{L^1(\mathbb{R})}<\infty$,  and since $\widetilde{\sigma}<-1$, it can be easily verified that $(x-b+1)^{\widetilde{\sigma}}\in L^p(b,\infty)$ for $1\leq p< \infty$. Taken all these together into account yields 
$\|g\|_{L^p(\mathbb{R})}=\|\mathbbm{1}_{\{ x>b\}} (\Gamma(\sigma))^{-1}  \RN{1}\|_{L^p(\mathbb{R})} = \|(\Gamma(\sigma))^{-1}  \RN{1}\|_{L^p(b,\infty)}<\infty$.
Therefore, $\mathbf{D}^{\mu} w \in L^p(\mathbb{R})$, which completes the proof.
\end{proof}
\begin{property}[Fourier Transform of R-L Derivatives  \cite{MR1347689}, p. 137] \label{lem:FTFD}
 If $\mu > 0$ and $w \in C_0^\infty(\mathbb{R})$, then
 \begin{equation} \label{eq:FTFDo}
 \mathcal{F}(\boldsymbol{D}^{\mu}w) = (2\pi i \xi)^{\mu} \mathcal{F}(w) \text{ and }
 \mathcal{F}(\boldsymbol{D}^{\mu*} w) = (-2\pi i \xi)^{\mu} \mathcal{F}(w), \quad \xi \ne 0, 
\end{equation}
where as in \Cref{prop:FTFI}, $(\mp i\xi)^{\sigma}$is understood as $|\xi|^\sigma e^{\mp  \sigma \pi i \cdot \emph{sign} (\xi)/2}$.
\end{property}
 \begin{proof}
Using a different version of the Fourier Transform (namely up to a sign $-2\pi$ in the exponential position), this property was stated without a proof in \cite{MR1347689}, p.137. The proof below is provided only for completeness. The proof is shown only for the Fourier Transform of left derivative since the right derivative counterpart can be carried out analogously. First notice that if $\mu $ is a positive integer, integration by parts and a simple calculation give the equality (see, e.g. \cite{MR924154}, p. 274).
Otherwise,  there is a positive integer $n$, such that $n-1<\mu< n$.
 Suppose $\text{supp}(w) \subset(a,b)$, with $-\infty<a<b<\infty$. Using \Cref{rem:ConvolutionofIntegral} and \Cref{thm:differentiability} gives
  \begin{equation}
 \boldsymbol{D}^{\mu}w =
\mathbbm{1}_{\{ x>a\}} \, \, {_aD_x^\mu}w
 = \mathbbm{1}_{\{ x>a\}} \, \,  {_{a}D_x^{-(n-\mu)}} w^{(n)}(x)
 =  \boldsymbol{D}^{-(n-\mu)} w^{(n)}(x).
 \end{equation}
 Notice now $0<n-\mu<1$, so using \Cref{prop:FTFI} yields
 \begin{equation}
 \mathcal{F}(\boldsymbol{D}^{\mu}w)=\mathcal{F}(\boldsymbol{D}^{-n+\mu} w^{(n)} ) =  (2\pi i\xi)^{-n+\mu} \mathcal{F}(w^{(n)}), \xi\neq 0.
 \end{equation}
The proof is completed by recalling that  $\mathcal{F}(w^{(n)}) = (2 \pi i \xi)^n \mathcal{F}(w)$.
 \end{proof}
\begin{property}\label{pro:TranslationDerivative} Given $w \in C_0^\infty(\mathbb{R})$, $\mu>0$, and a positive integer $n$  such that $n-1\leq \mu<n$, then
\begin{equation}
\begin{aligned}
\tau_h(\boldsymbol{D}^{\mu}w)=\boldsymbol{D}^{\mu}(\tau_hw)&,\quad \tau_h(\boldsymbol{D}^{\mu*}w)=\boldsymbol{D}^{\mu*}(\tau_hw)\\
\Pi_\kappa(\boldsymbol{D}^{\mu}w)=\kappa^{-\mu}\boldsymbol{D}^{\mu}(\Pi_\kappa w)&,\quad
\Pi_\kappa(\boldsymbol{D}^{\mu*}w)=\kappa^{-\mu}\boldsymbol{D}^{\mu*}(\Pi_\kappa w).
\end{aligned}
\end{equation}
\end{property}
\begin{proof}
First, since $w\in C_0^\infty(\mathbb{R})$, using \Cref{rem:ConvolutionofIntegral} and \Cref{thm:differentiability} gives $\boldsymbol{D}^\mu w= \boldsymbol{D}^{-(n-\mu)}w^{(n)}$. According to \Cref{pro:Translation}, 
\begin{equation*}
\tau_h(\boldsymbol{D}^\mu w)=\tau_h(\boldsymbol{D}^{-(n-\mu)}w^{(n)})=\boldsymbol{D}^{-(n-\mu)}(\tau_h w^{(n)})=\boldsymbol{D}^{-(n-\mu)}(\tau_h w)^{(n)}=\boldsymbol{D}^\mu(\tau_h w).
\end{equation*}
Similar argument is used to establish $\tau_h(\boldsymbol{D}^{\mu*}w)=\boldsymbol{D}^{\mu*}(\tau_hw)$.

Next, again using \Cref{pro:Translation},
\begin{equation}
\begin{aligned}
\Pi_\kappa(\boldsymbol{D}^{\mu}w)
&=\Pi_\kappa(\boldsymbol{D}^{-(n-\mu)}w^{(n)})\\
&=\kappa^{(n-\mu)}\boldsymbol{D}^{-(n-\mu)}(\Pi_\kappa w^{(n)})\\
&=\kappa^{(n-\mu)}\boldsymbol{D}^{-(n-\mu)}(\kappa^{-n}(\Pi_\kappa w)^{(n)})\\
&=\kappa^{-\mu}\boldsymbol{D}^{-(n-\mu)}((\Pi_\kappa w)^{(n)})\\
&=\kappa^{-\mu}\boldsymbol{D}^\mu (\Pi_\kappa w)
\end{aligned}
\end{equation}
Using similar argument, $\Pi_\kappa(\boldsymbol{D}^{\mu*}w)=\kappa^{-\mu} \boldsymbol{D}^{\mu*}(\Pi_\kappa w)$ follows.
\end{proof}
%
\section{Characterization of Fractional Sobolev Spaces}\label{sec:CFSS}
In this section, we shall characterize classical fractional Sobolev spaces $W^{s,2} (\mathbb{R})$ (namely $\widehat{H}^s(\mathbb{R})$) by giving another equivalent definition using weak fractional R-L derivatives. There is a vast amount of literatures devoted to Sobolev spaces (see, e.g., \cite{MR2759829}, \cite{MR2944369}, and \cite{MR2328004}), thus those well-established results pertaining to subsequent analyses are stated without proof.
%
\subsection{Some Facts on Sobolev Spaces}
\label{ssec:SWFSS}
%
\begin{definition}[Sobolev Spaces on $\mathbb{R}$ \cite{MR2597943}, p. 258, \cite{MR2424078} p. 250]
Let $m\in \mathbb{N}_0$, $1\leq p\leq \infty$. 
\begin{equation}
W^{m,p}(\mathbb{R})=\{w\in L^p(\mathbb{R}):D^\alpha w\in L^p(\mathbb{R}), \forall 0\leq \alpha \leq m\},
\end{equation} 
where $\alpha$ is integer and $D^\alpha u$ are  weak derivatives. If $s> 0$ is a real number and $m$ is the smallest integer greater than $s$, the fractional order Sobolev spaces are defined by complex interpolation as
\begin{equation}
W^{s, p}(\mathbb{R})=[L^p(\mathbb{R}), W^{m,p}(\mathbb{R})]_{s/m}.
\end{equation}
\end{definition}
\begin{definition}[Sobolev Spaces Via Fourier Transform, e.g. \cite{MR2944369,MR2328004}]\label{thm:FTHsR}
Given $s\geq 0$, let
\begin{equation}
\widehat{H}^s(\mathbb{R}) = \left \{ w \in L^2(\mathbb{R}) : \int_{\mathbb{R}} (1 + |2\pi\xi|^{2s}) |\widehat{w}(\xi) |^2 \, {\rm d} \xi < \infty \right \},
\end{equation}
where $\widehat{w}$ is Plancherel Transform defined in \Cref{thm:PAR}.
It is endowed with norm
\begin{equation}
\|w\|_{\widehat{H}^s (\mathbb{R})}:=\left(\|w\|^2_{L^2(\mathbb{R})} +|w|^2_{\widehat{H}^s(\mathbb{R})}\right)^{1/2}, ~~\text{with}~~
|w|_{\widehat{H}^s(\mathbb{R})}:=\|(2\pi\xi)^s \widehat{w}\|_{L^2(\mathbb{R})}.
\end{equation}
\end{definition}
\noindent
It is well-known that $\widehat{H}^s(\mathbb{R})$ is a Hilbert space.

\begin{theorem}[\cite{MR2328004}, p. 78]\label{thm:DensityOfSobolevSpaces}
$C_0^\infty(\mathbb{R})$ is  dense in $\widehat{H}^s(\mathbb{R})$.
\end{theorem}
\begin{theorem}[\cite{MR2424078}, p. 252]\label{thm:ImbedingTheorem}
Let $s\geq 0$. $u\in W^{s,2}(\mathbb{R})$ if and only if $u\in \widehat{H}^s(\mathbb{R})$. In addition,  $W^{s_1,2}(\mathbb{R})\subseteq{W^{s_2,2}(\mathbb{R})}$ for $s_1\geq s_2$.
\end{theorem}
\begin{remark}
Notice the  particular case, if $s=0$, $L^2(\mathbb{R})=\widehat{H}^0(\mathbb{R})=W^{0,2}(\mathbb{R})$.
\end{remark}
%
%
%
%
\subsection{Connections between Sobolev Spaces and R-L Derivatives}
First, a generalization of the usual integer-order weak derivatives to include weak fractional R-L derivatives is presented.
 \begin{definition}[Weak Fractional R-L Derivatives]\label{def:WFD}
 Let $s>0$, and $v, w \in L^1_{loc}(\mathbb{R})$.  The function $w$ is called weak $s$-order left fractional derivative of $v$, written as  $\boldsymbol{D}^s v = w$, provided
 \begin{equation}
(v, \boldsymbol{D}^{s* } \psi) = (w, \psi),~\forall \psi \in C_0^\infty(\mathbb{R}).
 \end{equation}
In a similar fashion, $w$ is weak $s$-order right fractional derivative of $v$, written as $\boldsymbol{D}^{s*} v=w$,  provided
 \begin{equation}
 (v, \boldsymbol{D}^{s}\psi) = (w, \psi), ~\forall \psi \in C_0^\infty(\mathbb{R}).
 \end{equation}
 \end{definition}
 \begin{lemma}[Uniqueness of Weak Fractional R-L Derivatives]
If $v\in L^1_{loc}(\mathbb{R})$ has a weak $s$-order left (or right) fractional derivative, then it is unique up to a set of zero measure.
 \begin{proof}
We only show the uniqueness for the left fractional derivative.  Assume $w_1,w_2\in L^1_{loc}(\mathbb{R})$ are both weak $s$-order fractional derivatives of $v$, namely,   $(w_1,\psi) = (v,\boldsymbol{D}^{s*} \psi) =(w_2,\psi)$, for all $\psi\in C_0^\infty(\mathbb{R})$. This implies $(w_1-w_2,\psi)=0$ for all $\psi\in C_0^\infty(\mathbb{R})$, whence $w_1=w_2$ a.e.. (e.g. \cite{MR2759829}, Corollary 4.24, p.110).
 \end{proof}
 \end{lemma}
   \begin{definition}\label{def:FractionalSobolevSpaces}
 Given $s\geq 0$, let
 \begin{equation*}
  \widetilde{W}^{s}_L(\mathbb{R})=\{v\in L^2(\mathbb{R}): \boldsymbol{D}^s v \in L^2(\mathbb{R})\}, ~~
  \widetilde{W}^{s}_R(\mathbb{R})=\{v\in L^2(\mathbb{R}): \boldsymbol{D}^{s*} v \in L^2(\mathbb{R})\}, 
 \end{equation*}
where $\boldsymbol{D}^s v$ and $\boldsymbol{D}^{s*} v$ are understood as the weak fractional derivative of \Cref{def:WFD}. A semi-norm
\begin{equation}
|v|_L:= \|\boldsymbol{D}^s v\|_{L^2(\mathbb{R})} ~~\text{for}~~\widetilde{W}^{s}_L(\mathbb{R}) ~~\text{and}~~ |v|_R:= \|\boldsymbol{D}^{s*} v\|_{L^2(\mathbb{R})}
~~\text{for}~~\widetilde{W}^{s}_R(\mathbb{R}),
\end{equation}
is given with the corresponding norm
$\quad \|v\|_{\star}:=(\|v\|^2_{L^2(\mathbb{R})}+|v|_\star^2)^{1/2}$, with $\star=L,R$.
 \end{definition}
\begin{remark}\label{rem:Notations}
By convention,  $\widetilde{W}^{0}_L(\mathbb{R})=\widetilde{W}^{0}_R(\mathbb{R})
=L^2(\mathbb{R})$. If $s$ is a positive integer, by definition,  $\boldsymbol{D}^s=D^s$, and $\boldsymbol{D}^{s*}=(-1)^sD^s$, so $\boldsymbol{D}^1=\boldsymbol{D}= D$, and $\boldsymbol{D}^{1*}=\boldsymbol{D}^*=-D$.  
\end{remark}

It is obvious that $\widetilde{W}^{s}_L(\mathbb{R})$ and $\widetilde{W}^{s}_R(\mathbb{R})$ are normed linear spaces. The following theorem describes a characterization of  Sobolev space $\widehat{H}^s(\mathbb{R})$ in terms of these spaces.
\begin{theorem}\label{thm:EquivalenceOfSpaces}
Given $s\geq 0$, $\widetilde{W}^{s}_L(\mathbb{R})$, $\widetilde{W}^{s}_R(\mathbb{R})$ and $\widehat{H}^s(\mathbb{R})$ are identical spaces with equal norms.
\end{theorem}
\begin{proof}
The proof only demonstrates $\widetilde{W}^{s}_L(\mathbb{R}) = \widehat{H}^s(\mathbb{R})$ and their norms equality, noting that the case for $\widetilde{W}^{s}_R(\mathbb{R}) = \widehat{H}^s(\mathbb{R})$ can be analogously established. By the construction of respective norms, equality of norms is achieved by showing equality of seminorms.

First we show that $\widehat{H}^s(\mathbb{R})\subseteq \widetilde{W}^{s}_L(\mathbb{R})$. Pick any $v\in \widehat{H}^s(\mathbb{R})$, which implies that $(2\pi i \xi)^s\widehat{v}\in L^2(\mathbb{R})$. In turn, this gives a justification for setting
$v_s:=((2\pi i \xi)^s \widehat{v})^{\vee}$, where $^\vee$ denotes the inverse of Plancherel Transform. Furthermore, \Cref{thm:PAR} (Plancherel) guarantees that $v_s\in L^2(\mathbb{R})$. An application of \Cref{thm:ParsevalFormula} gives
\begin{equation} \label{eq:mummyone}
(v, \boldsymbol{D}^{s*}\psi)=(\overline{v},\boldsymbol{D}^{s*}\psi)
=(\overline{\widehat{v}},\widehat{\boldsymbol{D}^{s*}\psi}), ~~\forall \psi \in C_0^\infty(\mathbb{R}).
\end{equation}
Next we use \Cref{lem:FTFD} to $\widehat{\boldsymbol{D}^{s*}\psi}$ and utilize \Cref{thm:ParsevalFormula} to yield
\begin{equation} \label{eq:mummytwo}
(\overline{\widehat{v}},\widehat{\boldsymbol{D}^{s*}\psi})
=(\overline{\widehat{v}},(-2\pi i\xi)^s\widehat{\psi})
=(\overline{(2\pi i\xi)^s\widehat{v}},\widehat{\psi})
=(v_s,\psi).
\end{equation}
By combining \cref{eq:mummytwo} with \cref{eq:mummyone}, we get $(v, \boldsymbol{D}^{s*}\psi) = (v_s,\psi)$ for any $\psi \in C_0^\infty(\mathbb{R})$, which according to \Cref{def:WFD} implies that
$\boldsymbol{D}^s v=v_s$, and thus $v \in \widetilde{W}^s_L(\mathbb{R})$. It is straightforward to see the equality of semi-norms, namely,
\[|v|_{\widehat{H}^s(\mathbb{R})}=\|(2\pi i\xi)^s\widehat{v}\|_{L^2(\mathbb{R})}=\|v_s\|_{L^2(\mathbb{R})}=|v|_L.\]

It remains now to show $\widehat{H}^s(\mathbb{R})\supseteq \widetilde{W}^{s}_L(\mathbb{R})$. Pick any $v\in \widetilde{W}^{s}_L(\mathbb{R})$. By \Cref{def:WFD},
\begin{equation}\label{equ:ReplacingFunction}
(v,\boldsymbol{D}^{s*} \psi)=(\boldsymbol{D}^s v,\psi), ~~\forall \psi
\in C_0^\infty(\mathbb{R})
\end{equation}
Fix $h\in \mathbb{R}$ and use $\tau_h\psi \in C_0^\infty(\mathbb{R})$ in \cref{equ:ReplacingFunction} to obtain
\begin{equation}
(v,\boldsymbol{D}^{s*} (\tau_h\psi))
=(v,\tau_h(\boldsymbol{D}^{s*}\psi))
=(\boldsymbol{D}^s v,\tau_h\psi),
\end{equation}
where \Cref{pro:TranslationDerivative} was used. For convenience, set 
\begin{equation}
A(z)=[\boldsymbol{D}^{s*}\psi](-z),\quad B(z)=[\boldsymbol{D}^s v](z), ~~\Psi(z) = \psi(-z).
\end{equation}
Then
\begin{equation}
\int_\mathbb{R} v(t)  A(h-t) \, {\rm d}t
=\int_{\mathbb{R}}B(t) \Psi(h-t)\, {\rm d}t, ~~ \forall h\in \mathbb{R},
\end{equation}
or in other words,
\begin{equation}\label{equ:ConvolutionCommunicateFourier}
[v*A](h)=[B*\Psi](h). 
\end{equation}
Notice that $v, \Psi \in L^2(\mathbb{R})$, and $A, B \in L^1(\mathbb{R})$ by \Cref{prop:Boundedness}.  \Cref{thm:ConvolutionFourierTransform} applied to \cref{equ:ConvolutionCommunicateFourier} gives
 \begin{equation}\label{equ:CrossCorrelation}
\widehat{v}\widehat{A}=\widehat{B}\widehat{\Psi}.
\end{equation} 
Since $\Psi(z) = \psi(-z)$ and $\psi \in C_0^\infty(\mathbb{R})$, then $\widehat{\Psi} = \overline{\widehat{\psi}} = \overline{\mathcal{F}(\psi)}$. In a similar fashion, and using \Cref{lem:FTFD}, $\widehat{A} = \overline{\widehat{\boldsymbol{D}^{s*}\psi}} =  \overline{(-2\pi i\xi)^s\mathcal{F}(\psi)} = (2\pi i\xi)^s \overline{\mathcal{F}(\psi)}$.
Putting these back to \cref{equ:CrossCorrelation} gives
\begin{equation}\label{equ:Nonzero}
\left( (2\pi i\xi)^s\widehat{v}-\widehat{\boldsymbol{D}^s v}\right)\overline{\mathcal{F}(\psi)}=0~~\forall \psi \in C_0^\infty(\mathbb{R}).
\end{equation}
We claim that \cref{equ:Nonzero} implies that $(2\pi i\xi)^s\widehat{v}=\widehat{\boldsymbol{D}^s v}$. Since $\psi \in C_0^\infty(\mathbb{R})$ is arbitrary, choose a non-zero $\psi$ such that by \Cref{thm:PAR} (Plancherel), $\mathcal{F}(\psi)\neq 0$. Without loss of generality, suppose $[\mathcal{F}(\psi)](a)\neq 0$ at point $a\in \mathbb{R}$. Since $\mathcal{F}(\psi)$ is continuous, there exists an open interval $(c,d)$ containing $a$ such that $\mathcal{F}(\psi)\neq 0$ in $(c,d)$. Notice that
$\Pi_\epsilon \psi \in C_0^\infty(\mathbb{R})$ and by the Fourier Transform property of dilation operator,  $[\mathcal{F}(\Pi_\epsilon \psi)](\xi) = \epsilon^{-1} [\mathcal{F}(\psi)](\epsilon^{-1}\xi)$ for arbitrary $\epsilon>0$. This means $\mathcal{F}(\Pi_\epsilon \psi) \neq 0$ in $(\epsilon c, \epsilon d)$. With this fact in place and using $\Pi_\epsilon \psi$ as a test function in \cref{equ:Nonzero} implies
\begin{equation}
(2\pi i\xi)^s\widehat{v}-\widehat{\boldsymbol{D}^s v}=0,\quad \text{in $ (\epsilon c, \epsilon d)$}.
\end{equation}
Since $\epsilon$ is arbitrary, we conclude that
\begin{equation}\label{equ:ThelastOne}
(2\pi i\xi)^s\widehat{v}=\widehat{\boldsymbol{D}^s v},\quad \text{in $\mathbb{R}$}.
\end{equation}
Therefore $(2\pi i\xi)^s\widehat{v}\in L^2(\mathbb{R})$, and thus $v\in \widehat{H}^s(\mathbb{R})$. Furthermore, \cref{equ:ThelastOne} implies $|v|_{\widehat{H}^s(\mathbb{R})}=|v|_L$.
\end{proof}
The preceding theorem reveals that $\boldsymbol{D}^sv$ and $\boldsymbol{D}^{s*}v$ always makes sense for $v\in \widehat{H}^s(\mathbb{R})$, $s>0$. The following results will be utilized later.
\begin{corollary}\label{cor:DensityOfSobolevsSpaces}
$C_0^\infty(\mathbb{R})$ is dense in $\widetilde{W}^s_L(\mathbb{R})$ and  $\widetilde{W}^s_R(\mathbb{R})$.
\end{corollary}
\begin{proof}
This is a consequence of \Cref{thm:EquivalenceOfSpaces} and \Cref{thm:DensityOfSobolevSpaces}.
\end{proof}
\begin{corollary}\label{cor:DensityOfC}
$v\in \widehat{H}^{s}(\mathbb{R})$ if and only if there exists a  sequence $\{v_n\}\subset C_0^\infty(\mathbb{R})$ such that $\{v_n\}, \{\boldsymbol{D}^sv_n\}$ are Cauchy sequences in $L^2(\mathbb{R})$ with $\lim_{n\rightarrow \infty}v_n=v$.
Likewise,
$v\in \widehat{H}^{s}(\mathbb{R}) $ if and only if there exists a  sequence $\{v_n\}\subset C_0^\infty(\mathbb{R})$ such that $\{v_n\}, \{\boldsymbol{D}^{s*}v_n\}$ are Cauchy sequences in $L^2(\mathbb{R})$, with $\lim_{n\rightarrow \infty}v_n=v$.
\end{corollary}
\begin{proof}
This is a consequence of \Cref{thm:EquivalenceOfSpaces} and \Cref{cor:DensityOfSobolevsSpaces}.
\end{proof}
\begin{remark}
As a consequence of \Cref{cor:DensityOfC}, $\lim_{n\rightarrow \infty}\boldsymbol{D}^{s}v_n=\boldsymbol{D}^{s}v$ and $\lim_{n\rightarrow\infty}\boldsymbol{D}^{s*}v_n=\boldsymbol{D}^{s*}v$.
\end{remark}
%
\section{Stationary Fractional Diffusion-Advection-Reaction Equations} \label{sec:SFDARE}
In this section, we investigate the following Stationary Fractional Diffusion-Advection-Reaction equation: find $u\in \widehat{H}^{2-\mu}(\mathbb{R}) $ such that
\begin{equation}\label{eq:FODE}
\begin{cases}
[Lu](x)=f(x),\quad x\in \mathbb{R},\\
\text{where}\,Lu = p\boldsymbol{D}^{2-\mu} u+q\boldsymbol{D}^{(2-\mu)*}u+aDu+bu, f\in L^2(\mathbb{R}),\\
\text{with } \,p,q, a, b, \mu \in \mathbb{R}, \text{ such that } b \ne 0, p^2+q^2\ne0, \mu \in (0,1).
\end{cases}
\end{equation}
In this equation, $\boldsymbol{D}^{2-\mu} u, \boldsymbol{D}^{(2-\mu)*}u, Du$ are all understood as weak derivatives. The condition $p^2+q^2\neq 0$ implies that at least either $p\boldsymbol{D}^{2-\mu} u$ or  $q\boldsymbol{D}^{(2-\mu)*}u$  must be present in \cref{eq:FODE}, thereby avoiding the classical first order ODEs. Also, we point out that $b\neq 0$ plays an important role in determining the regularity of solution to problem~\cref{eq:FODE}. The main results are stated in \Cref{thm:1stex}, \Cref{Thm:Regularity}. 

\subsection{Several Important Tools}
Several results that are crucial in the subsequent analysis are first established.
 \begin{theorem}\label{thm:symmetry}
For $v, w \in C_0^\infty(\mathbb{R})$ and $\mu \geq 0$, it is true that
\begin{equation} \label{eq:MSK}
\begin{aligned}
& (\boldsymbol{D}^{\mu} v , \boldsymbol{D}^{\mu} w )  =
  (\boldsymbol{D}^{\mu*} v , \boldsymbol{D}^{\mu*} w ) = (2\pi)^{2\mu}\int_\mathbb{R}|\xi|^{2\mu} \widehat{v}(\xi) \overline{\widehat{w}(\xi)}  \,{\rm d}\xi,\\
 &(\boldsymbol{D}^{\mu} v, \boldsymbol{D}^{\mu*} w ) +(\boldsymbol{D}^{\mu} w, \boldsymbol{D}^{\mu*} v )
 =2\cos(\mu \pi)(\boldsymbol{D}^{\mu} v , \boldsymbol{D}^{\mu} w ).
 \end{aligned}
 \end{equation}
 \end{theorem}
\begin{proof}
The two equalities in \cref{eq:MSK} are true when $\mu=0$, so suppose $\mu>0$.
Since $v,w \in C_0^\infty(\mathbb{R})$, \Cref{prop:Boundedness} guarantees that $\boldsymbol{D}^\mu v , \boldsymbol{D}^\mu w, \boldsymbol{D}^{\mu*}v, \boldsymbol{D}^{\mu*}w \in L^p(\mathbb{R})$ with $p\geq 1$. Using \Cref{thm:ParsevalFormula} (Parseval Formula) and  in combination with \Cref{lem:FTFD}  give
\begin{equation*}
\begin{aligned}
(\boldsymbol{D}^{\mu} v, \boldsymbol{D}^{\mu} w ) &=
(\mathcal{F}(\boldsymbol{D}^\mu v), \overline{\mathcal{F}(\boldsymbol{D}^\mu w)} ) = (2\pi)^{2\mu}\int_\mathbb{R}|\xi|^{2\mu} \widehat{v}(\xi) \overline{\widehat{w}(\xi)}  \,{\rm d}\xi, \\
(\boldsymbol{D}^{\mu*} v, \boldsymbol{D}^{\mu*} w ) &=
(\mathcal{F}(\boldsymbol{D}^{\mu*} v), \overline{\mathcal{F}(\boldsymbol{D}^{\mu*} w)} ) = (2\pi)^{2\mu}\int_\mathbb{R}|\xi|^{2\mu} \widehat{v}(\xi) \overline{\widehat{w}(\xi)}  \,{\rm d}\xi,
\end{aligned}
\end{equation*}
confirming the first equality in \cref{eq:MSK}. In a similar fashion,
\begin{equation}\label{equ:ExcellentResults}
\begin{aligned}
(\boldsymbol{D}^{\mu} v , \boldsymbol{D}^{\mu*} w ) &=
(\mathcal{F}(\boldsymbol{D}^{\mu} v), \overline{\mathcal{F}(\boldsymbol{D}^{\mu*} w)} )
= (2\pi)^{2\mu}  \RN{1},\\
(\boldsymbol{D}^{\mu} w , \boldsymbol{D}^{\mu*} v ) &=
(\overline{\mathcal{F}(\boldsymbol{D}^{\mu} w)}, \mathcal{F}(\boldsymbol{D}^{\mu*} v) )
= (2\pi)^{2\mu}  \RN{2},
\end{aligned}
\end{equation}
where
\begin{equation*}
\RN{1} =  \int_\mathbb{R} (i\xi)^{\mu} \, \overline{(-i\xi)^\mu} \, \widehat{v}(\xi) \,\overline{\widehat{w}(\xi)} \,{\rm d}\xi ~~\text{and}~~ \RN{2} =   \int_\mathbb{R} \overline{(i\xi)^{\mu}} \, (-i\xi)^\mu \,\widehat{v}(\xi) \, \overline{\widehat{w}(\xi)} \,{\rm d}\xi.
\end{equation*}
Upon utilization of \Cref{rem:ComplexPowerFunctions},
\begin{equation*}
\begin{aligned}
\RN{1} &= 
\int_\mathbb{R} |\xi|^{2\mu} e^{i \mu \text{sign}(\xi) \pi/2} \, \overline{e^{-i \mu \text{sign}(\xi) \pi/2}} \, \widehat{v}(\xi) \,\overline{\widehat{w}(\xi)} \,{\rm d}\xi
= 
\int_\mathbb{R} |\xi|^{2\mu} e^{i \mu \text{sign}(\xi) \pi}  \, \widehat{v}(\xi) \,\overline{\widehat{w}(\xi)} \,{\rm d}\xi,
\end{aligned}
\end{equation*}
\begin{equation*}
\begin{aligned}
\RN{2} &= \hspace*{-0.1cm}
\int_\mathbb{R} |\xi|^{2\mu} \overline{e^{i \mu \text{sign}(\xi) \pi/2}} \, e^{-i \mu \text{sign}(\xi) \pi/2} \, \widehat{v}(\xi) \,\overline{\widehat{w}(\xi)} \,{\rm d}\xi
\hspace*{-0.05cm} = \hspace*{-0.1cm}
\int_\mathbb{R} |\xi|^{2\mu} e^{-i \mu \text{sign}(\xi) \pi}  \, \widehat{v}(\xi) \,\overline{\widehat{w}(\xi)} \,{\rm d}\xi.
\end{aligned}
\end{equation*}
Summation of $\RN{1}$ and $\RN{2}$ and decomposition of $\mathbb{R}$ into $(-\infty,0)$ and $(0,\infty)$ yield
\begin{equation*}
\begin{aligned}
\RN{1} + \RN{2} &=(e^{-\pi i\mu}+e^{\pi i\mu})\int_{-\infty}^0 |\xi|^{2\mu}\, \widehat{v}(\xi) \,\overline{\widehat{w}(\xi)} \,{\rm d}\xi+(e^{-\pi i\mu}+e^{\pi i\mu}) \int_0^\infty |\xi|^{2\mu}\, \widehat{v}(\xi) \,\overline{\widehat{w}(\xi)} \,{\rm d}\xi\\
&=2\cos(\mu \pi)\int_{\mathbb{R}}|\xi|^{2\mu} \widehat{v}(\xi) \overline{\widehat{w}(\xi)}  \,{\rm d}\xi,
\end{aligned}
\end{equation*}
from which the second equality in \cref{eq:MSK} follows.
\end{proof}%
 \begin{lemma}\label{lem:dense}
The set $M=\{w: w=Lv, ~\forall v\in C_0^\infty(\mathbb{R})$\} is dense in $L^{2}(\mathbb{R})$. Furthermore, the set
$\widetilde{M}=\{w: w=\widetilde{L} v, ~\forall v\in C_0^\infty(\mathbb{R})$\} is also dense in $L^{2}(\mathbb{R})$, where
$\widetilde{L} v = p\boldsymbol{D}^{(2-\mu)*}v+q\boldsymbol{D}^{2-\mu}v-aDv+bv$.
\end{lemma}
\begin{proof}
By \Cref{prop:Boundedness}, $M\subset L^2(\mathbb{R})$. Since $L^2(\mathbb{R})$ is a Hilbert space, the density of $M$ is established by invoking \Cref{Thm:dense}.  Furthermore, because $C_0^\infty(\mathbb{R})$ is closed under addition and scalar multiplication, so is $M$, and thus $M$ is a subspace of $L^2(\mathbb{R})$. Therefore all conditions are met for the utilization of \Cref{Thm:dense}.
 Using \Cref{lem:FTFD} (Fourier Transform) for $w=Lv$ gives 
\[
[\mathcal{F}(w)](\xi)= H(\xi) [\mathcal{F}(v)](\xi), ~~H(\xi) = \left(p(2\pi i \xi)^{2-\mu}+q(-2\pi i\xi)^{2-\mu}+a(2\pi i\xi)+b \right).
\]
Setting $\displaystyle \vartheta = \frac{(2-\mu) \pi \, \text{sign}(\xi) }{2}$ and following \Cref{rem:ComplexPowerFunctions}, $H(\xi)$ is expressed as
\begin{equation}
\begin{aligned}
H(\xi) &= (2\pi |\xi|)^{2-\mu} \left( p e^{i \vartheta} + q e^{-i \vartheta} \right)
+ a(2\pi i\xi)+b\\
&= \left( (2\pi |\xi|)^{2-\mu} (p+q) \cos(\vartheta) + b \right) +
i \left( (2\pi |\xi|)^{2-\mu} (p-q) \sin(\vartheta) + 2\pi a \xi \right).
\end{aligned}
\end{equation}
If $H(\xi) = 0$, then $\xi$ must satisfy
\begin{equation} 
\begin{cases}
(2\pi |\xi|)^{2-\mu} (p+q) \cos(\vartheta) + b=0,\\
(2\pi |\xi|)^{2-\mu} (p-q) \sin(\vartheta) + 2\pi a \xi =0.
\end{cases}
\end{equation}
Notice that $\cos(\vartheta)$ and $\sin(\vartheta)$ can never be zero when $\mu \in (0,1)$.
In such a case, there is at most one $\xi \in \mathbb{R}$ such that $H(\xi) = 0$,  thereby confirming that $H(\xi) \neq 0$ a.e in $\mathbb{R}$.

At this stage, we repeat some of the arguments in the proof of \Cref{thm:EquivalenceOfSpaces}. Specifically, choose $0 \ne \varphi \in C_0^\infty(\mathbb{R})$, so that by \Cref{thm:PAR} (Plancherel), $\mathcal{F}(\varphi)\neq 0$. On the account of continuity of $\mathcal{F}(\varphi)$, there exists $(a,b)\subset \mathbb{R}$ such that $\mathcal{F}(\varphi)\neq 0$ in $(a,b)$. Choose
$\epsilon>0$, and let $v \in C_0^\infty(\mathbb{R})$ such that $v=\Pi_\epsilon \varphi$. 
It is true that  $[\mathcal{F}(v)](\xi)= \epsilon^{-1} [\mathcal{F}(\varphi)](\epsilon^{-1}\xi)$ and  thus $\mathcal{F}(v)\neq 0$ in $(\epsilon a, \epsilon b)$. This and in combination with the fact that $H(\xi) \ne 0$ a.e. in $\mathbb{R}$ implies $\mathcal{F}(w) \neq 0$ a.e. in $(\epsilon a, \epsilon b)$, or equivalently,  $\overline{\mathcal{F}(w)} \neq 0$ a.e. in $(\epsilon a, \epsilon b)$.
 
Let $g\in L^2(\mathbb{R})$ such that $(g,w) = 0 ~\text{ for any } w\in M$. By \Cref{Thm:dense}, the density of $M$ is confirmed if this equation implies that $g = 0$.
Given $w \in M$ and any fixed $y\in \mathbb{R}$, and using the translation operator, set  

\[
G(y) = (g, \tau_y w) = \int_{\mathbb{R}} g(x) w(x-y) \, {\rm d} x =  \int_{\mathbb{R}} g(y-z) w(-z) \, {\rm d} z, 
\] 
where a change of variable was used to get the last term in the above equality.
Notice that \Cref{pro:TranslationDerivative} implies that
$ \tau_y w= \tau_y L v = L(\tau_y v)$, where it is true that $\tau _y v\in C_0^\infty(\mathbb{R})$ for $v\in C_0^\infty(\mathbb{R})$. This means $\tau_y w\in M$ and thus $G(y)=0$ for every $y\in \mathbb{R}$. This fact along with an application of \Cref{thm:ConvolutionFourierTransform} yields $0 = \widehat{g}\overline{\widehat{w}} = \widehat{g} \overline{\mathcal{F}(w)}$. However, as noted earlier, $\overline{\mathcal{F}(w)} \neq 0$ a.e. in $(\epsilon a, \epsilon b)$, so it must be that $\widehat{g}=0$ a.e.  in $(\epsilon a,\epsilon b)$.  Because $\epsilon>0$ is arbitrary, $\widehat{g}=0$ in any open interval, and thus $\widehat{g}=0$ in $\mathbb{R}$. Another use of  \Cref{thm:PAR} (Plancherel) concludes that $g=0$, implying the density of $M$ in $L^2(\mathbb{R})$.

Density of $\widetilde{M}$ is shown by repeating the foregoing arguments using $\widetilde{L}$.
 \end{proof}

\Cref{lem:dense} is the basis for computing $\| w \|_{L^2(\mathbb{R})}$ for $w\in M$ whose representation can be either
$w=L v$ or $w = L(\Pi_{1/\delta} v)$,  $v \in C_0^\infty(\mathbb{R})$. The results are stated in
\Cref{lem:norm} and \Cref{lem:normtwo}.
 \begin{lemma}\label{lem:norm}
 For $w=Lv$, with $v\in C_0^\infty(\mathbb{R})$, the following norm equality holds,
\[
 \|w\|^2_{L^2(\mathbb{R})} = \sum_{j=1}^5 C_j \|\boldsymbol{D}^{\sigma_j}v\|^2_{L^2(\mathbb{R})},
\]
where
\begin{equation}
\begin{aligned}
C_1&=p^2+q^2+2pq\cos(\sigma_1\pi), & \sigma_1 &= 2-\mu,\\
C_2&=2a(q-p)\cos(\sigma_2\pi), & \sigma_2 &= \frac{1}{2}(3-\mu),\\
C_3&=a^2, & \sigma_3 &= 1,\\
C_4&=2b(p+q)\cos(\sigma_4\pi), & \sigma_4 &= \frac{1}{2}(2-\mu),\\
C_5&=b^2, & \sigma_5 &= 0.
\end{aligned}
\end{equation}
 \end{lemma}
 \begin{proof}
 By definition,
\begin{equation}
\|w\|_2^2=(Lv ,Lv)=\RN{1}+\RN{2}+\RN{3},
\end{equation} 
 where
 \begin{equation}
\begin{aligned}
 \RN{1}&=(p\boldsymbol{D}^{2-\mu} v+q\boldsymbol{D}^{(2-\mu)*}v\, ,\,p\boldsymbol{D}^{2-\mu} v+q\boldsymbol{D}^{(2-\mu)*}v),\\
 \RN{2}&=(aDv+b v,aDv+b v),\\
 \RN{3}&=2(p\boldsymbol{D}^{2-\mu} v+q\boldsymbol{D}^{(2-\mu)*}v\, ,\, aDv+b v).
\end{aligned}
 \end{equation}
In the following, we compute $\RN{1},\RN{2},\RN{3}$ separately. The idea is that we would like to shift the exponents in the fractional derivatives by using basic properties of R-L operators, so that \Cref{thm:symmetry} can be utilized.

Application of \Cref{thm:symmetry} shows that
 \begin{equation}
 \begin{aligned}
 \RN{1}&=(p\boldsymbol{D}^{\sigma_1} v,p\boldsymbol{D}^{\sigma_1} v)+(q\boldsymbol{D}^{\sigma_1*}v,q\boldsymbol{D}^{\sigma_1*}v)+2(p\boldsymbol{D}^{\sigma_1} v,q\boldsymbol{D}^{\sigma_1*}v)\\
 &=(p^2+q^2)(\boldsymbol{D}^{\sigma_1} v,\boldsymbol{D}^{\sigma_1} v)+2pq\cos((2-\mu)\pi)(\boldsymbol{D}^{\sigma_1} v,\boldsymbol{D}^{\sigma_1} v)\\
 &=C_1 \|\boldsymbol{D}^{\sigma_1} v\|^2_{L^2(\mathbb{R})}.
 \end{aligned}
 \end{equation}

An integration by parts shows that $(aDv,bv) = -(av, Dv)$ and thus $(aDv,v) = 0$. This means
 \begin{equation}
 \RN{2}=(aDv,aDv)+(bv,bv)+2(aDv,bv) = C_3 \|\boldsymbol{D}^{\sigma_3}v\|^2_{L^2(\mathbb{R})}+ C_5\|\boldsymbol{D}^{\sigma_5} v\|^2_{L^2(\mathbb{R})}.
 \end{equation}

Moreover, we make a decomposition $\RN{3} = 2b \RN{3}_1 + 2a\RN{3}_2$, with
\[
\RN{3}_1 = p(\boldsymbol{D}^{2-\mu} v, v)+q(\boldsymbol{D}^{(2-\mu)*}v, v)~\text{and}~ \RN{3}_2 =p(\boldsymbol{D}^{2-\mu} v,Dv)+q(\boldsymbol{D}^{(2-\mu)*}v,Dv).
\]
The following calculation for $\RN{3}_1$ is performed:
\[
\begin{aligned}
\RN{3}_1 &= p(\boldsymbol{D}^{-\mu}v^{(2)}, v)+q(\boldsymbol{D}^{-\mu*}v^{(2)},v)  \hspace*{1.84cm} (\text{by \Cref{rem:ConvolutionofIntegral} and \Cref{thm:differentiability}})\\
&= p(\boldsymbol{D}^{-\mu/2}\boldsymbol{D}^{-\mu/2}v^{(2)}, v)+q(\boldsymbol{D}^{-\mu/2*}\boldsymbol{D}^{-\mu/2*}v^{(2)},v) \hspace*{1.648cm}  (\text{by \Cref{cor:SGInfinity}})\\
&= p(\boldsymbol{D}^{-\mu/2}v^{(2)}, \boldsymbol{D}^{-\mu/2*}v)+q(\boldsymbol{D}^{-\mu/2*}v^{(2)},\boldsymbol{D}^{-\mu/2}v) \hspace*{1.648cm} (\text{by \Cref{cor:ADJInfinity}})\\
&= p(D^2\boldsymbol{D}^{-\mu/2}v, \boldsymbol{D}^{-\mu/2*}v)+q(D^2\boldsymbol{D}^{-\mu/2*}v,\boldsymbol{D}^{-\mu/2}v) \hspace*{0.385cm}  (\text{by \Cref{rem:ConvolutionofIntegral} and \ref{thm:differentiability}})\\
&= p(\boldsymbol{D}^{\sigma_4}v,\boldsymbol{D}^{\sigma_4*}v)+q(\boldsymbol{D}^{\sigma_4*}v,\boldsymbol{D}^{\sigma_4}v) \hspace*{1.569cm} (\text{int. by parts and \Cref{def:RLD}})\\
&= (p+q) \cos(\sigma_4 \pi) \|\boldsymbol{D}^{\sigma_4}v\|^2_{L^2(\mathbb{R})}.  \hspace*{4.8cm}(\text{by \Cref{thm:symmetry}})
\end{aligned}
\]
Similar calculation is performed for $\RN{3}_2$,
 after first integrating it by parts and using \Cref{def:RLD}:
\[
\begin{aligned}
\RN{3}_2 &=-p(\boldsymbol{D}^{3-\mu}v,v)+q(\boldsymbol{D}^{(3-\mu) *}v, v)  \hspace*{1.212cm} (\text{int. by parts and \Cref{def:RLD}}) \\
&=-p(\boldsymbol{D}^{-\mu} v^{(3)}, v)+q(-\boldsymbol{D}^{-\mu *}v^{(3)},v)  \hspace*{1.915cm} (\text{by \Cref{rem:ConvolutionofIntegral} and \ref{thm:differentiability}}) \\
  &=-p(v^{(3)}, \boldsymbol{D}^{-\mu*}v)+q(-v^{(3)},\boldsymbol{D}^{-\mu}v)  \hspace*{3.015cm} (\text{by \Cref{cor:ADJInfinity}})\\
   &=-p(\boldsymbol{D}^{-1}v^{(4)},\boldsymbol{D}^{-\mu*}v)+q(\boldsymbol{D}^{-1*}v^{(4)}, \boldsymbol{D}^{-\mu}v) \hspace*{1.71cm} (\text{by \Cref{cor:IPC}}) \\
&=-p(\boldsymbol{D}^{-(1-\mu)/2}\boldsymbol{D}^{-(1+\mu)/2}v^{(4)},\boldsymbol{D}^{-\mu*}v)\\
&\hspace*{0.4cm}+q(\boldsymbol{D}^{-(1-\mu)/2*}\boldsymbol{D}^
   {-(1+\mu)/2*}v^{(4)}, \boldsymbol{D}^{-\mu}v) \hspace*{2.367cm} (\text{by \Cref{cor:SGInfinity}})\\
&=-p(\boldsymbol{D}^{-(1+\mu)/2}v^{(4)},\boldsymbol{D}^{-(1+\mu)/2*}v\\
      &\hspace*{0.4cm}+q(\boldsymbol{D}^{-(1+\mu)/2*}v^{(4)}, \boldsymbol{D}^{-(1+\mu)/2}v) \hspace*{3.26cm} (\text{by \Cref{cor:ADJInfinity}})\\
 &=-p(D^4\boldsymbol{D}^{-(1+\mu)/2}v,\boldsymbol{D}^{-(1+\mu)/2*}v)\\
 &\hspace*{0.4cm}+q(D^4\boldsymbol{D}^{-(1+\mu)/2*}v, \boldsymbol{D}^{-(1+\mu)/2}v) \hspace*{2.08cm}   (\text{by \Cref{rem:ConvolutionofIntegral} and \ref{thm:differentiability}})\\
  &=-p(\boldsymbol{D}^{\sigma_2}v,\boldsymbol{D}^{\sigma_2*}v)+q(\boldsymbol{D}^{\sigma_2*}v,\boldsymbol{D}^{\sigma_2}v) \hspace*{0.43cm}  (\text{int. by parts and \Cref{def:RLD}})\\
   &=(q-p) \cos(\sigma_2 \pi) \|\boldsymbol{D}^{\sigma_2}v\|^2_{L^2(\mathbb{R})}. \hspace*{3.95cm}  (\text{by \Cref{thm:symmetry}})
\end{aligned}
\]
This completes the proof.
 \end{proof}
 
 \begin{remark} \label{rem:Cs}
 It is worth noting that from the construction $C_1>0$ (because $|\cos(\sigma_1\pi)|<1, p^2 + q^2 \ne 0$), $C_3 \ge 0$, $C_5 > 0$ (because $b \ne 0$), $\cos(\sigma_2\pi)<0$, and $\cos(\sigma_4\pi)<0$. However  $C_2, C_4$ may be negative, and $C_2\ge 0$ only when $a(q-p)\le0$, $C_4\ge 0$ only when $a(p+q)\le0$. Discussion on different cases of $C_2, C_4$ is relegated to \Cref{lem:normcase1}, \Cref{lem:normcase2} and \Cref{lem:normcase3}.
 \end{remark}
 
 \begin{lemma}\label{lem:normtwo}
Let $\delta>0$, $\varphi \in C_0^\infty(\mathbb{R})$ and $w = L(\Pi_{1/\delta} \varphi)$. Then
  \begin{equation}\label{equ:W_nFourier}
 \begin{aligned}
  \|w\|^2_{L^2(\mathbb{R})} &= \sum_{j=1}^5  \frac{C_j}{\delta^{2\sigma_j-1}} 
\int_\mathbb{R}|2\pi \xi|^{2\sigma_j}|\widehat{\varphi}|^2\, {\rm d}\xi.
 \end{aligned}
 \end{equation}
 \end{lemma}
 
 \begin{proof}
 By \Cref{pro:TranslationDerivative}, $[\boldsymbol{D}^{\sigma_j}(\Pi_{1/\delta} \varphi)](x) =  \delta^{-\sigma_j} [\Pi_{1/\delta} (\boldsymbol{D}^{\sigma_j}\varphi)](x) = [\boldsymbol{D}^{\sigma_j}\varphi](x/\delta)$, so using \Cref{lem:norm} along with appropriate change of variable in the integration yields
\[
 \|w\|^2_{L^2(\mathbb{R})} = \sum_{j=1}^5 C_j \|\boldsymbol{D}^{\sigma_j}(\Pi_{1/\delta} \varphi)\|^2_{L^2(\mathbb{R})}
  = \sum_{j=1}^5  \frac{C_j}{\delta^{2\sigma_j-1}}  \| \boldsymbol{D}^{\sigma_j}\varphi\|^2_{L^2(\mathbb{R})},
  \]
  from which \cref{equ:W_nFourier} is obtained through application of
\Cref{lem:FTFD} (Fourier Transform) and  \Cref{thm:PAR} (Plancherel).
 \end{proof}
 
Recall from \Cref{rem:Cs}, $C_1, C_3, C_5$ are non-negative, however $C_2, C_4$ may be positive or non-positive, and this presents a constraint in guaranteeing the existence of solutions to \cref{eq:FODE}. Therefore, different cases for $C_2, C_4$ are treated separately to help materialize the conclusion in \Cref{thm:1stex}. \Cref{lem:normcase1}, \Cref{lem:normcase2}, and \Cref{lem:normcase3} below show different representations of norm of $w=L((\Pi_{1/\alpha} \varphi))$ according to different cases of $C_2, C_4$. More precisely, we discuss three different cases:
 \[(1)~ C_2\geq 0, C_4<0,\quad (2)~ C_2<0, C_4\geq 0,\quad (3)~ C_2<0, C_4<0.\]
 The case $C_2\ge 0, C_4\ge 0$ is treated in a straightforward manner later on.
 \begin{lemma} \label{lem:normcase1}
With $\{C_j,\sigma_j\}_{j=1}^5$ defined in \Cref{lem:norm}, assume $C_2\ge0$, $C_4<0$ and $b^2 > -C_4\,\alpha^{2(\sigma_5-\sigma_4)}$, where  $\alpha>0$ satisfies
\begin{equation} \label{eq:sum1Cj}
\sum_{j=1}^4 \frac{C_j}{\alpha^{2\sigma_j-1}} > 0.
\end{equation}
Then for $w = L(\Pi_{1/\alpha} \varphi)$, with $\varphi \in C_0^\infty(\mathbb{R})$, we have
\begin{equation}
\begin{aligned}
  \|w\|^2_{L^2(\mathbb{R})} &= \sum_{\substack{\ell=1\\ \ell \ne 4}}^5 \RN{1}_\ell + \sum_{\ell=1}^3 \RN{2}_\ell + \RN{3}_1,
 \end{aligned}
 \end{equation}
where
\begin{equation} \label{eq:Icase1}
\begin{aligned}
\RN{1}_\ell &=  Q_{1,\ell} \int_\mathbb{R}|2\pi \xi|^{2\sigma_\ell}|\widehat{\varphi}|^2\, {\rm d}\xi, ~~\ell=1,2,3,5,\\
\RN{2}_\ell &= Q_{2,\ell} \left(  \int_\mathbb{R}|2\pi \xi|^{2\sigma_\ell}|\widehat{\varphi}|^2\, {\rm d}\xi -  \int_{|\xi |>1}|2\pi \xi|^{2\sigma_4}|\widehat{\varphi}|^2\, {\rm d}\xi \right), ~~\ell=1,2,3,\\
\RN{3}_1 &= Q_{3,1} \left(\int_\mathbb{R}|2\pi \xi|^{2\sigma_5}|\widehat{\varphi}|^2\, {\rm d}\xi - \int_{|\xi|\le 1}|2\pi \xi|^{2\sigma_4}|\widehat{\varphi}|^2\, {\rm d}\xi \right),\\
\end{aligned}
\end{equation}
with constants $Q_{1,1}>0$, $Q_{1,5}>0$ and the rest of $Q_{j,\ell}\ge0$. 
 \end{lemma}
 \begin{proof}
 Since $0<C_1, 0<2\sigma_1-1=\max_{j=1,2,3,4}\{2\sigma_j-1\}$, there always exists a sufficiently small positive number $\alpha$ such that inequality~\cref{eq:sum1Cj} holds true.  With this $\alpha$, condition $b^2 > -C_4\,\alpha^{2(\sigma_5-\sigma_4)}$ equivalently implies $\frac{C_5}{\alpha^{2\sigma_5-1}} + \frac{C_4}{\alpha^{2\sigma_4-1}}  > 0$.
Since $C_4<0$, from \cref{eq:sum1Cj}, there exist non-positive numbers $A_1,A_2,A_3$ such that
\[
\sum_{j=1}^3 A_j =\frac{C_4}{\alpha^{2\sigma_4-1}}, \quad \frac{C_1}{\alpha^{2\sigma_1-1}}+A_1>0,  \quad \frac{C_2}{\alpha^{2\sigma_2-1}}+A_2 \ge 0, \quad \frac{C_3}{\alpha^{2\sigma_3-1}}+A_3\ge0.
\]
With $\alpha$ in place of $\delta$ and $\frac{C_4}{\alpha^{2\sigma_4-1}}$ substituted by $\sum_{j=1}^3 A_j$ in \Cref{lem:normtwo}, and by adding and subtracting appropriate terms, one has
\[
\begin{aligned}
  \|w\|^2_{L^2(\mathbb{R})} &= \sum_{\substack{j=1,\\ j\ne4}}^5  \frac{C_j}{\alpha^{2\sigma_j-1}} 
\int_\mathbb{R}|2\pi \xi|^{2\sigma_j}|\widehat{\varphi}|^2\, {\rm d}\xi
+ \left(\sum_{j=1}^3 A_j \right) \int_\mathbb{R}|2\pi \xi|^{2\sigma_4}|\widehat{\varphi}|^2\, {\rm d}\xi\\
&= \sum_{\substack{\ell=1 \\ \ell \ne 4}}^5 \RN{1}_\ell + \sum_{\ell=1}^3 \RN{2}_\ell +  \RN{3}_1,
 \end{aligned}
\]
where all the terms are as in \cref{eq:Icase1}, with
\[
\begin{aligned}
Q_{1,\ell} &= \frac{C_\ell}{\alpha^{2\sigma_\ell-1}}+A_\ell, ~\ell=1,2,3, ~~ Q_{1,5} = \sum_{j=4}^5 \frac{C_j}{\alpha^{2\sigma_j-1}},\\
Q_{2,\ell} &= -A_\ell, ~\ell=1,2,3,\\
Q_{3,1} &= -\frac{C_4}{\alpha^{2\sigma_4-1}}.
\end{aligned}
\]
\end{proof}
 
 %
 \begin{lemma} \label{lem:normcase2}
With $\{C_j,\sigma_j\}_{j=1}^5$ defined in \Cref{lem:norm}, assume $C_2<0$, $C_4\ge0$, and 
$b^2>-\sum_{j=2}^4C_j\,\alpha^{2(\sigma_5-\sigma_j)}$, where $\alpha>0$ satisfies
\begin{equation} \label{eq:sum2Cj}
\sum_{j=1}^2 \frac{C_j}{\alpha^{2\sigma_j-1}} > 0.
\end{equation}
Then for $w = L(\Pi_{1/\alpha} \varphi)$, with  $\varphi \in C_0^\infty(\mathbb{R})$, we have
\begin{equation}
\begin{aligned}
  \|w\|^2_{L^2(\mathbb{R})} &= \sum_{\substack{\ell=1 \\ \ell \ne 2}}^5 \RN{1}_\ell + \sum_{\ell=3,4,5} \RN{2}_\ell + \RN{3}_1,
 \end{aligned}
 \end{equation}
where
\begin{equation} \label{eq:Icase2}
\begin{aligned}
\RN{1}_\ell &=  Q_{1,\ell} \int_\mathbb{R}|2\pi \xi|^{2\sigma_\ell}|\widehat{\varphi}|^2\, {\rm d}\xi, ~~\ell=1,3,4,5,\\
\RN{2}_\ell &= Q_{2,\ell} \left(  \int_\mathbb{R}|2\pi \xi|^{2\sigma_\ell}|\widehat{\varphi}|^2\, {\rm d}\xi -  \int_{|\xi |\le1}|2\pi \xi|^{2\sigma_2}|\widehat{\varphi}|^2\, {\rm d}\xi \right), ~~\ell=3,4,5,\\
\RN{3}_1 &= Q_{3,1} \left(\int_\mathbb{R}|2\pi \xi|^{2\sigma_1}|\widehat{\varphi}|^2\, {\rm d}\xi - \int_{|\xi|> 1}|2\pi \xi|^{2\sigma_2}|\widehat{\varphi}|^2\, {\rm d}\xi \right),\\
\end{aligned}
\end{equation}
with constants $Q_{1,1}>0$, $Q_{1,5}>0$ and the rest of $Q_{j,\ell}\ge0$.
 \end{lemma}
 \begin{proof}
 Since $0<C_1, 0<2\sigma_1-1=\max_{j=1,2}\{2\sigma_j-1\}$, there always exists a sufficiently small positive number $\alpha$ such that \cref{eq:sum2Cj} holds true.
  With this $\alpha$, condition $b^2>-\sum_{j=2}^4C_j\,\alpha^{2(\sigma_5-\sigma_j)}$ equivalently implies
  \begin{equation}\label{equ:TheSecondLemmaDecomposition}
   \sum_{j=2}^5 \frac{C_j}{\alpha^{2\sigma_j-1}}  > 0,
  \end{equation}
Since $C_2<0$, from inequality~\cref{equ:TheSecondLemmaDecomposition}, there exist non-positive numbers $B_3,B_4,B_5$ such that
\[
\sum_{j=3}^5 B_j = \dfrac{C_2}{\alpha^{2\sigma_2-1}}, \quad  \dfrac{C_3}{\alpha^{2\sigma_3-1}}+B_3 \geq 0, \quad \dfrac{C_4}{\alpha^{2\sigma_4-1}}+B_4\geq 0, \quad \frac{C_5}{\alpha^{2\sigma_5-1}}+B_5>0.
\]
With $\alpha$ in place of $\delta$ and $\frac{C_2}{\alpha^{2\sigma_2-1}}$ substituted by $\sum_{j=3}^5 B_j$ in \Cref{lem:normtwo}, and by adding and subtracting appropriate terms, one has
\[
\begin{aligned}
  \|w\|^2_{L^2(\mathbb{R})} &= \sum_{\substack{j=1,\\ j\ne2}}^5  \frac{C_j}{\alpha^{2\sigma_j-1}} 
\int_\mathbb{R}|2\pi \xi|^{2\sigma_j}|\widehat{\varphi}|^2\, {\rm d}\xi
+ \left(\sum_{j=3}^5 B_j \right) \int_\mathbb{R}|2\pi \xi|^{2\sigma_2}|\widehat{\varphi}|^2\, {\rm d}\xi\\
&= \sum_{\substack{\ell=1 \\ \ell \ne 2}}^5 \RN{1}_\ell + \sum_{\ell=3,4,5} \RN{2}_\ell + \RN{3}_1,
 \end{aligned}
 \]
where all the terms are as in \cref{eq:Icase2}, with
\[
\begin{aligned}
Q_{1,1} &= \sum_{\ell=1}^2 \frac{C_\ell}{\alpha^{2\sigma_\ell-1}},~~
Q_{1,\ell} = \dfrac{C_\ell}{\alpha^{2\sigma_\ell-1}}+B_\ell, ~\ell=3,4,5,\\
Q_{2,\ell} &= -B_\ell, ~\ell=3,4,5,\\
Q_{3,1} &= -\frac{C_2}{\alpha^{2\sigma_2-1}}.
\end{aligned}
\]
 \end{proof} 
 %
 \begin{lemma} \label{lem:normcase3}
With $\{C_j,\sigma_j\}_{j=1}^5$ defined in \Cref{lem:norm}, assume $C_2<0$, $C_4<0$, and

$b^2> -\sum_{j=2,4}C_j\, \alpha^{2(\sigma_5-\sigma_j)} $, where $\alpha>0$ satisfies
\begin{equation} \label{eq:sum3Cj}
\sum_{j=1,2,4} \frac{C_j}{\alpha^{2\sigma_j-1}} > 0.
\end{equation}
Then for $w = L(\Pi_{1/\alpha} \varphi)$, with $\varphi \in C_0^\infty(\mathbb{R})$, we have
\begin{equation}
\begin{aligned}
  \|w\|^2_{L^2(\mathbb{R})} &= \sum_{\ell=1,3,5} \RN{1}_\ell + \sum_{\ell=2,4} \RN{2}_\ell + \sum_{\ell=2,4} \RN{3}_\ell,
 \end{aligned}
 \end{equation}
where
\begin{equation} \label{eq:Icase3}
\begin{aligned}
\RN{1}_\ell &=  Q_{1,\ell} \int_\mathbb{R}|2\pi \xi|^{2\sigma_\ell}|\widehat{\varphi}|^2\, {\rm d}\xi, ~~\ell=1,3,5,\\
\RN{2}_\ell &= Q_{2,\ell} \left(  \int_\mathbb{R}|2\pi \xi|^{2\sigma_1}|\widehat{\varphi}|^2\, {\rm d}\xi -  \int_{|\xi | \ge1}|2\pi \xi|^{2\sigma_\ell}|\widehat{\varphi}|^2\, {\rm d}\xi \right), ~~\ell=2,4,\\
\RN{3}_\ell &= Q_{3,\ell} \left(\int_\mathbb{R}|2\pi \xi|^{2\sigma_5}|\widehat{\varphi}|^2\, {\rm d}\xi - \int_{|\xi|< 1}|2\pi \xi|^{2\sigma_\ell}|\widehat{\varphi}|^2\, {\rm d}\xi \right), ~~\ell=2,4,
\end{aligned}
\end{equation}
with constants $Q_{1,1}>0$, $Q_{1,5}>0$ and the rest of $Q_{j,\ell}\ge0$. 
 \end{lemma}
 \begin{proof}
 Since $0< C_1, 0<2\sigma_1-1=\max_{ j=1,2,4}\{2\sigma_j-1\}$, there always exists a sufficiently small positive number $\alpha$ such that \cref{eq:sum3Cj} holds true. With this $\alpha$, condition $b^2> -\sum_{j=2,4}C_j\, \alpha^{2(\sigma_5-\sigma_j)} $ equivalently implies $\sum_{j=2,4,5} \frac{C_j}{\alpha^{2\sigma_j-1}}  > 0$.
With $\alpha$ in place of $\delta$ in \Cref{lem:normtwo}, and by adding and subtracting appropriate terms, one has
\begin{equation}
\begin{aligned}
  \|w\|^2_{L^2(\mathbb{R})} &= \sum_{\substack{j=1,\\ j\ne2,4}}^5  \frac{C_j}{\alpha^{2\sigma_j-1}} 
\int_\mathbb{R}|2\pi \xi|^{2\sigma_j}|\widehat{\varphi}|^2\, {\rm d}\xi
+  \sum_{\substack{j=2,4}} \frac{C_j}{\alpha^{2\sigma_j-1}}  \int_\mathbb{R}|2\pi \xi|^{2\sigma_j}|\widehat{\varphi}|^2\, {\rm d}\xi  \\
&= \sum_{\ell=1,3,5} \RN{1}_\ell + \sum_{\ell=2,4} \RN{2}_\ell + \sum_{\ell=2,4} \RN{3}_\ell,
 \end{aligned}
 \end{equation}
where all the terms are as in \cref{eq:Icase3}, with
\[
\begin{aligned}
Q_{1,1} &= \sum_{j=1,2,4} \frac{C_j}{\alpha^{2\sigma_j-1}}, ~~ 
Q_{1,3} = \frac{C_3}{\alpha^{2\sigma_3-1}}, ~~
Q_{1,5} &= \sum_{j=2,4,5} \frac{C_j}{\alpha^{2\sigma_j-1}},\\
Q_{2,\ell} &= -\dfrac{C_\ell}{\alpha^{2\sigma_\ell-1}}, ~\ell=2,4,\\
Q_{3,\ell} &= -\dfrac{C_\ell}{\alpha^{2\sigma_\ell-1}}, ~\ell=2,4.
\end{aligned}
\]
 \end{proof}
 
 Notice that in view of $\sigma_1>\cdots>\sigma_5$, \Cref{lem:normcase1} implies that each $\RN{1}_\ell$, $\RN{2}_\ell$, $\RN{3}_1$ is at least non-negative. The same situation applies to \Cref{lem:normcase2} and \Cref{lem:normcase3}.

\subsection{Existence, Uniqueness, and Regularity of the Solution}
At this stage, we are ready to prove the existence and uniqueness of strong  solutions to problem~\cref{eq:FODE}. The following theorem implies that, roughly speaking, if $|b|$ is big enough compared to other coefficients, there always exists a unique solution $u\in \widehat{H}^{2-\mu}(\mathbb{R})$ to problem~\cref{eq:FODE}. 

\begin{theorem} \label{thm:1stex}
Consider problem~ \cref{eq:FODE} with $\{C_j,\sigma_j\}_{j=1}^5$  defined in \Cref{lem:norm}. For either of the following cases,
\begin{itemize}
\item [(i)] $C_2 \ge 0$, $C_4 \ge 0$ or
\item [(ii)] $C_2 \ge 0$, $C_4 < 0$, $b^2 > -C_4\,\alpha^{2(\sigma_5-\sigma_4)}$, and $\alpha>0$ with $\displaystyle \sum_{j=1}^4 \frac{C_j}{\alpha^{2\sigma_j-1}} > 0$, or
\item [(iii)] $C_2 < 0$, $C_4 \ge 0$, $b^2>-\displaystyle\sum_{j=2}^4C_j\,\alpha^{2(\sigma_5-\sigma_j)}$, and $\alpha>0$ with $\displaystyle \sum_{j=1}^2 \frac{C_j}{\alpha^{2\sigma_j-1}} > 0$, or
\item [(iv)] $C_2<0$, $C_4<0$, $b^2> -\displaystyle \sum_{j=2,4}C_j\, \alpha^{2(\sigma_5-\sigma_j)} $, and $\alpha>0$ with $\displaystyle \sum_{j=1,2,4} \frac{C_j}{\alpha^{2\sigma_j-1}} > 0$,
\end{itemize}
there exists a unique solution $u\in \widehat{H}^{2-\mu}(\mathbb{R})$ that satisfies \cref{eq:FODE}. Furthermore,
\begin{equation}
 \|u\|_{\widehat{H}^{2-\mu}(\mathbb{R})}\leq C\|f\|_{L^2(\mathbb{R})},
 \end{equation}
 for some positive constants $C>0$ depending only on $L$. 
\end{theorem}
\begin{proof}
These four different cases are discussed together in a unified way as follows.

Fix $f \in L^2(\mathbb{R})$ in \cref{eq:FODE}. 
Since $b\neq 0$ in \cref{eq:FODE}, for each case in the theorem, \Cref{lem:dense} guarantees that there is a Cauchy sequence $\{w_n\} \subset M$ in $L^2(\mathbb{R})$ such that
\begin{equation}\label{equ:limit}
\lim_{n\rightarrow \infty}\|w_n-f\|^2_{L^2(\mathbb{R})}=0,
\end{equation}
 where $w_n=Lu_n$ for certain sequence $\{u_n\}\subset{C_0^\infty(\mathbb{R})}$. 
Now we intend to show that equation~\cref{equ:limit} implies both $\{u_n\},\{\boldsymbol{D}^{2-\mu}u_n\}$ are actually Cauchy sequences in $L^2(\mathbb{R})$ under each case in \Cref{thm:1stex}. 

 To do so, we compute $ \|w_n-w_m\|^2_{L^2(\mathbb{R})}$ in the following in terms of $\{\varphi_n\}$, where $\{\varphi_n\}=\{\Pi_{\beta}u_n\}$ (namely rewrite sequence $\{u_n\}$ as $\{\Pi_{1/\beta}\varphi_n\}$), with  $\beta>0$  chosen in this way: $\beta=1$ for case $(i)$, $\beta=\alpha$ for cases $(ii), (iii), (iv)$.
Let us note carefully that case $(i), (ii), (iii), (iv)$ allow us to apply \Cref{lem:norm}, \Cref{lem:normcase1}, \Cref{lem:normcase2} and  \Cref{lem:normcase3} accordingly. By doing so, we have
 \begin{equation}\label{eq:W_n}
 \begin{aligned}
 \|w_n-w_m\|^2_{L^2(\mathbb{R})}
 &= \sum_{\ell=1,5} P_{1,\ell} \int_\mathbb{R}|2\pi \xi|^{2\sigma_\ell}|\widehat{\varphi}_n-\widehat{\varphi}_m|^2\, {\rm d}\xi + \text{Remainder},
 \end{aligned}
 \end{equation}
where $\text{Remainder}\ge 0$, while $P_{1,1}$ and $P_{1,5}$ are both strictly positive, with
\[
P_{1,1} =
\begin{cases}
\displaystyle C_1 & \text{for case \emph{(i)}},\\

Q_{1,1} &  \text{for case \emph{(ii)}},\\

Q_{1,1} &  \text{for case \emph{(iii)}},\\

Q_{1,1}  & \text{for case \emph{(iv)}},
\end{cases}
~~P_{1,5} =
\begin{cases}
\displaystyle C_5 & \text{for case \emph{(i)}},\\

Q_{1,5} &  \text{for case \emph{(ii)}},\\

Q_{1,5}&  \text{for case \emph{(iii)}},\\

Q_{1,5}  & \text{for case \emph{(iv)}}.
\end{cases}
\]
Given any $\epsilon>0$, there exists a positive integer $N$ such that, for $n,m>N$, it is true that  $\|w_n-w_m\|^2_{L^2(\mathbb{R})} < \epsilon$.
Since every term in equation~\cref{eq:W_n} is nonnegative with the first two are positive, it means 
\[
P_{1,\ell} \int_\mathbb{R}|2\pi \xi|^{2\sigma_\ell}|\widehat{\varphi}_n-\widehat{\varphi}_m|^2\, {\rm d}\xi < \epsilon, ~~ \ell=1,5. 
\]
Based on the fact that
\begin{equation}\label{equ:ConvertingTerm}
\|\boldsymbol{D}^{\sigma_\ell}u_n-\boldsymbol{D}^{\sigma_\ell}u_m\|^2_{L^2(\mathbb{R})} =
\dfrac{1}{\beta^{2\sigma_\ell-1}}\int_{\mathbb{R}}|2\pi \xi|^{2\sigma_\ell}|\widehat{\varphi}_n-\widehat{\varphi}_m|^2\, {\rm d}\xi,  ~~ \ell=1,5,
\end{equation}
it is concluded that
\[
\|\boldsymbol{D}^{\sigma_\ell}u_n-\boldsymbol{D}^{\sigma_\ell}u_m\|^2_{L^2(\mathbb{R})} < \frac{\epsilon}{\beta^{2\sigma_\ell-1} \, P_{1,\ell}}, ~~ \ell=1,5. 
\]
Recall that $\sigma_1 = 2 - \mu$ and $\sigma_5 = 0$, so this last inequality implies that 
$\{\boldsymbol{D}^{2-\mu}u_n\}$ and $\{u_n\}$ are Cauchy sequences for each case in \Cref{thm:1stex}. 
Denoting the limit by
 \[
 u=\lim_{n\rightarrow \infty}u_n,
 \]
  then \Cref{cor:DensityOfC} gives $u\in \widehat{H}^{2-\mu}(\mathbb{R})$.
 Furthermore, \Cref{thm:ImbedingTheorem} and \Cref{thm:EquivalenceOfSpaces} guarantee the existence of $\boldsymbol{D}^{2-\mu}u$, $\boldsymbol{D}^{(2-\mu)*}u$, and $Du$, therefore $u$ is the solution of \cref{eq:FODE}. Actually, by revisiting \cref{equ:limit}, it is seen that
 \begin{equation}
 \begin{aligned}
 f&=\lim_{n\rightarrow \infty} w_n\\
 &=\lim_{n\rightarrow \infty} \left( p\boldsymbol{D}^{2-\mu} u_n+q\boldsymbol{D}^{(2-\mu)*}u_n+aDu_n+b u_n \right)\\
 &=p\boldsymbol{D}^{2-\mu} u+q\boldsymbol{D}^{(2-\mu)*}u+aDu+b u.
 \end{aligned}
 \end{equation}

To estimate the norm of $u$, we revisit \cref{eq:W_n} to get
\[
 \begin{aligned}
 \|w_n\|^2_{L^2(\mathbb{R})}&= \sum_{j=1}^5 C_j \|\boldsymbol{D}^{\sigma_j} u_n\|^2_{L^2(\mathbb{R})}
 = \sum_{\ell=1,5} P_{1,\ell} \int_\mathbb{R}|2\pi \xi|^{2\sigma_\ell}|\widehat{\varphi}_n|^2\, {\rm d}\xi + \text{Remainder}.
 \end{aligned}
 \]
 Simply by noticing that $\text{Remainder}\ge 0$, while $P_{1,1}>0$, $P_{1,5}>0$ and using equation~\cref{equ:ConvertingTerm}, we obtain
 \begin{equation}
 \begin{aligned}
 \|w_n\|^2_{L^2(\mathbb{R})} &\ge \sum_{\ell=1,5} P_{1,\ell} \int_\mathbb{R}|2\pi \xi|^{2\sigma_\ell}|\widehat{\varphi}_n|^2\, {\rm d}\xi\\
 &=  \sum_{\ell=1,5}  \beta^{2\sigma_\ell-1} P_{1,\ell} \|\boldsymbol{D}^{\sigma_\ell}u_n\|^2_{L^2(\mathbb{R})}\\
 &\ge \frac{1}{C^2} \sum_{\ell=1,5}   \|\boldsymbol{D}^{\sigma_\ell}u_n\|^2_{L^2(\mathbb{R})},
 \end{aligned}
 \end{equation}
 where $\displaystyle \frac{1}{C} = \left(\min_{\ell=1,5}\{  \beta^{2\sigma_\ell-1} P_{1,\ell} \}\right)^{1/2}$. By taking the limit as $n\rightarrow \infty$, the last inequality produces 
 \[
 \sum_{\ell=1,5}   \|\boldsymbol{D}^{\sigma_\ell}u\|^2_{L^2(\mathbb{R})}\leq C^2\|f\|^2_{L^2(\mathbb{R})}.
 \]
 Taking the root at both sides, by \Cref{def:FractionalSobolevSpaces}  and \Cref{thm:EquivalenceOfSpaces}, we get
\begin{equation}\label{equ:NormEstimateEquation}
 \|u\|_{\widehat{H}^{2-\mu}(\mathbb{R})}\leq C\|f\|_{L^2(\mathbb{R})}.
 \end{equation}

 For the uniqueness of solution, to the contrary, let $u_1, u_2 \in \widehat{H}^{2-\mu}(\mathbb{R})$ be solutions of \cref{eq:FODE} under each same case. This means $L(u_1-u_2)=0$, which by the stability estimate~\cref{equ:NormEstimateEquation} yields $\|u_1-u_2||_{\widehat{H}^{2-\mu}(\mathbb{R})} =0$, implying $u_1=u_2$ a.e., hence the uniqueness of the solution of \cref{eq:FODE}.
This completes the whole proof.
\end{proof}
 \begin{remark}
 A closer look of the above proof indicates that \Cref{thm:1stex} can be established for ordinary differential equations that use $\widetilde{L}$ (see \Cref{lem:dense}). This is because each $\{C_i, \sigma_i\}$ (see \Cref{lem:norm}) corresponding to $\widetilde{L}$ is equal to the one obtained for $L$.
 \end{remark}

Once we have established the existence of solutions, now we are ready to discuss the regularity of solutions, it turns out that the smoothness of solutions are exactly determined by the source function $f$.
 \begin{theorem}\label{Thm:Regularity}
Under the same condition in \Cref{thm:1stex} and if $f\in \widehat{H}^{m}(\mathbb{R})$, then there is a unique $u\in \widehat{H}^{2-\mu +m}(\mathbb{R})$, where $m\in \mathbb{N}_0$ and
 \begin{equation} \label{eq:Regularity}
 \|u\|_{\widehat{H}^{2-\mu+m}(\mathbb{R})}\leq C\|f\|_{\widehat{H}^{m}(\mathbb{R})},
 \end{equation}
for some positive constant $C$ depending only on $L$.
 \end{theorem}
 \begin{proof}
This theorem is established by induction on $m$, noting that the case $m=0$ has been proven in \Cref{thm:1stex} ($\widehat{H}^0(\mathbb{R})=L^2(\mathbb{R})$ by convention). Assume the statement of theorem is true for a positive integer $m$.

Let $f\in \widehat{H}^{m+1}(\mathbb{R})$, which means $f, Df\in \widehat{H}^{m}(\mathbb{R})$. By the induction assumption, there are $u, v \in \widehat{H}^{2-\mu+m}(\mathbb{R})$ such that 
\begin{equation}\label{equ:Df}
Lu=f,\quad Lv=Df.
\end{equation}
Furthermore, using \Cref{def:WFD},
\begin{equation}\label{equ:UsingDefinitionOfWFD}
(f, \boldsymbol{D}^{*} \psi)=(Df,\psi), \forall \psi \in C_0^\infty(\mathbb{R}) \quad\text{(Recall $\boldsymbol{D}^{*} $ from \Cref{rem:Notations})}.
\end{equation} 
In the following, the intention is to demonstrate that actually $u\in \widehat{H}^{3-\mu+m}(\mathbb{R})$. 

Since $u, v \in \widehat{H}^{2-\mu+m}(\mathbb{R})$, by \Cref{cor:DensityOfC} and \Cref{thm:ImbedingTheorem}, there exist  sequences $\{u_n\},\{v_n\}\subset C_0^\infty(\mathbb{R})$  such that 
\begin{equation}\label{equ:TakingLimit}
\begin{aligned}
&\lim_{n\rightarrow \infty} \| u_n-u \|_{L^2(\mathbb{R})} = 0 ,\quad \lim_{n\rightarrow \infty} \| \boldsymbol{D}^su_n-\boldsymbol{D}^{s}u \|_{L^2(\mathbb{R})} = 0, ~~\forall s\in [0,2-\mu], \\
&\lim_{n\rightarrow \infty} \| v_n-v \|_{L^2(\mathbb{R})} = 0,\quad
\lim_{n\rightarrow \infty} \| \boldsymbol{D}^sv_n-\boldsymbol{D}^{s}v \|_{L^2(\mathbb{R})} = 0, ~~\forall s\in [0,2-\mu].
\end{aligned}
\end{equation}
Convergence of these sequences justifies the following equalities:
\begin{equation} \label{equ:TakingTheLimits}
\begin{aligned}
(f,\boldsymbol{D}^{*}\psi) &= (Lu,\boldsymbol{D}^{*}\psi) 
= \lim_{n \rightarrow \infty} (Lu_n,\boldsymbol{D}^{*}\psi) ~~\forall \psi \in C_0^\infty(\mathbb{R}),\\
(Df,\psi) &= (Lv,\boldsymbol{D}^{*}\psi) 
= \lim_{n \rightarrow \infty} (Lv_n,\boldsymbol{D}^{*}\psi) ~~\forall \psi \in C_0^\infty(\mathbb{R}).
\end{aligned}
\end{equation}
Using the definition of $L$,
\begin{equation*}
\begin{aligned}
(Lu_n,\boldsymbol{D}^{*}\psi)
&= p(\boldsymbol{D}^{-\mu}u_n^{(2)},\boldsymbol{D}^*\psi)
-q(\boldsymbol{D}^{-\mu*}u_n^{(2)},D\psi)\\
&\hspace*{0.5cm}-a(Du_n,D\psi)
-b(u_n,D\psi) \hspace*{.6cm}  \text{(by \Cref{rem:ConvolutionofIntegral} \& 
\Cref{thm:differentiability})}\\
&=p(D u_n, \boldsymbol{D}^{(2-\mu)*}\psi)
+q(Du_n,\boldsymbol{D}^{2-\mu}\psi)\\
&\hspace*{0.5cm}-a(Du_n,D\psi)+b(Du_n,\psi). \hspace*{.5cm} \text{(int. by parts \& \Cref{cor:ADJInfinity})} 
\end{aligned}
\end{equation*}
Taking limit as $n \rightarrow \infty$ of this last equation and using the first equality in \cref{equ:TakingTheLimits} give
\begin{equation} \label{eq:mm03}
\begin{aligned}
(f,\boldsymbol{D}^{*}\psi) &= p(D u, \boldsymbol{D}^{(2-\mu)*}\psi)
+q(Du,\boldsymbol{D}^{2-\mu}\psi)-a(Du,D\psi)+b(Du,\psi)\\
&= (D u, p\boldsymbol{D}^{(2-\mu)*}\psi
+q\boldsymbol{D}^{2-\mu}\psi-aD\psi+b\psi)\\
&= (D u, \widetilde{L} \psi), ~~\forall \psi \in C_0^\infty(\mathbb{R}),
\end{aligned}
\end{equation}
where $\widetilde{L}$ is as defined in \Cref{lem:dense}.
Similarly, second equality in \cref{equ:TakingTheLimits} yields
\begin{equation} \label{eq:mm04}
\begin{aligned}
(Df,\psi) &= (Lv,\psi)
&= (v,\widetilde{L}\psi), ~~\forall \psi \in C_0^\infty(\mathbb{R}).
\end{aligned}
\end{equation}
We substitute \cref{eq:mm03} and \cref{eq:mm04} back into \cref{equ:UsingDefinitionOfWFD} to obtain
\begin{equation}\label{equ:MoreInformation}
(D u-v, \widetilde{L} \psi) = 0,~~\forall \psi \in C_0^\infty(\mathbb{R}).
\end{equation}
Since \Cref{lem:dense} confirms that $\widetilde{L}\psi$ is dense in $L^2(\mathbb{R})$, it is concluded that $Du-v=0$ or $Du=v$, and thus
$Du\in \widehat{H}^{2-\mu+m}(\mathbb{R})$.
This means there is $w \in L^2(\mathbb{R})$ such that $(Du,\boldsymbol{D}^{(\sigma+m)*}\psi)=(w, \psi)$ for any $\psi\in C_0^\infty(\mathbb{R})$.
Furthermore,
\begin{equation}
\begin{aligned}
(Du,\boldsymbol{D}^{(2-\mu+m)*}\psi) &=\lim_{n\rightarrow \infty}( Du_n,\boldsymbol{D}^{(2-\mu+m)*}\psi) \\
&=\lim_{n\rightarrow \infty}(u_n,\boldsymbol{D}^{(3-\mu+m)*}\psi)\\
&=(u,\boldsymbol{D}^{(3-\mu+m)*}\psi).
\end{aligned}
\end{equation}
Therefore $(u,\boldsymbol{D}^{(3-\mu+m)*}\psi)=(w, \psi), \forall \psi\in C_0^\infty(\mathbb{R})$, which by \Cref{thm:EquivalenceOfSpaces}, implies $u\in\widehat{H}^{3-\mu+m}(\mathbb{R})$.

To establish the estimate, assumption in the induction argument gives
\begin{equation}
\|u\|_{\widehat{H}^{2-\mu+m}(\mathbb{R})}\leq C_1\|f\|_{\widehat{H}^{m}(\mathbb{R})},\quad\|v\|_{\widehat{H}^{2-\mu+m}(\mathbb{R})}\leq C_2\|Df\|_{\widehat{H}^{m}(\mathbb{R})},
\end{equation}
for  certain positive constants $C_1, C_2$. By norms equality stated in \Cref{thm:EquivalenceOfSpaces}, 
\begin{equation}\label{ineq:NormInequality1}
\|u\|_{L^2(\mathbb{R})}^2+\|\boldsymbol{D}^{2-\mu+m}u\|_{L^2(\mathbb{R})}^2\le C_1^2\left(\|f\|_{L^2(\mathbb{R})}^2+\|D^m f\|_{L^2(\mathbb{R})}^2\right),
\end{equation}
and
\begin{equation}\label{ineq:NormInequality2}
\|v\|_{L^2(\mathbb{R})}^2+\|\boldsymbol{D}^{2-\mu+m}v\|_{L^2(\mathbb{R})}^2\le C_2^2\left(\|Df\|_{L^2(\mathbb{R})}^2+\|D^m (Df)\|_{L^2(\mathbb{R})}^2\right).
\end{equation}
Since $v=Du$, \cref{ineq:NormInequality1} and \cref{ineq:NormInequality2} can be used to give
\begin{equation*}
\begin{aligned}
\|u\|_{L^2(\mathbb{R})}^2&+\|\boldsymbol{D}^{3-\mu+m}u\|_{L^2(\mathbb{R})}^2\\
 &=\|u\|_{L^2(\mathbb{R})}^2+\|\boldsymbol{D}^{2-\mu+m}v\|_{L^2(\mathbb{R})}^2\\
&\le C_1^2\left(\|f\|_{L^2(\mathbb{R})}^2+\|D^m f\|_{L^2(\mathbb{R})}^2\right)+
C_2^2\left(\|Df\|_{L^2(\mathbb{R})}^2+\|D^m (Df)\|_{L^2(\mathbb{R})}^2\right)\\
&=C_1^2 \|f\|^2_{\widehat{H}^{m}(\mathbb{R})}+ C_2^2 \|Df\|^2_{\widehat{H}^{m}(\mathbb{R})}\\
&\le (C_1^2 + C_2^2) \int_{\mathbb{R}} (1+|2\pi \xi|^{2m})(1+|2\pi\xi|^2) | \widehat{f}(\xi) |^2 \, {\rm d} \xi\\
&= (C_1^2 + C_2^2) (J_1 + J_2),
\end{aligned}
\end{equation*}
where 
\[
\begin{aligned}
J_1 &= \int_{|2\pi\xi| <1} \hspace*{-0.3cm} (1+|2\pi \xi|^{2m})(1+|2\pi\xi|^2) | \widehat{f}(\xi) |^2 \, {\rm d} \xi \le
\int_{|2\pi\xi| <1}  \hspace*{-0.3cm}  (3+|2\pi \xi|^{2m+2}) | \widehat{f}(\xi) |^2 \, {\rm d} \xi, \\
J_1 &= \int_{|2\pi\xi| \ge 1}  \hspace*{-0.3cm}  (1+|2\pi \xi|^{2m})(1+|2\pi\xi|^2) | \widehat{f}(\xi) |^2 \, {\rm d} \xi \le
\int_{|2\pi\xi| \ge1}  \hspace*{-0.3cm} (1+3|2\pi \xi|^{2m+2}) | \widehat{f}(\xi) |^2 \, {\rm d} \xi.
\end{aligned}
\]
Therefore
\begin{equation}\label{equ:NormEstimateForHigherRegulatiry}
\|u\|_{\widehat{H}^{3-\mu+m}(\mathbb{R})}<\sqrt{3(C_1^2+C_2^2)}\|f\|_{\widehat{H}^{m+1}(\mathbb{R})}.
\end{equation}
The uniqueness of solutions directly follows from \cref{equ:NormEstimateForHigherRegulatiry} as was done in  \Cref{thm:1stex}.
 \end{proof}

A closer look at the proof of \Cref{Thm:Regularity} (see \cref{equ:MoreInformation}) reveals a possibility for a stronger conclusion, namely for $f\in \widehat{H}^m(\mathbb{R})$, $L(D^nu)=D^nf$, where $n\in \mathbb{N}_0$ and $0\leq n\leq m$.
By repeated application of \Cref{Thm:Regularity} for $m=0,1,2,\cdots$, infinite differentiability of $u$ can be deduced as follows.

 \begin{corollary}
Under hypothesis of \Cref{thm:1stex} and if $f\in C^\infty(\mathbb{R})$, then $u\in C^\infty(\mathbb{R})$. 
 \end{corollary}
 \begin{proof}
 Since $u\in \widehat{H}^{2-\mu+m}(\mathbb{R})$ for $m=0,1,2,\cdots$ by \Cref{Thm:Regularity}, Sobolev Embedding Theorem (\cite{MR1787146}, p.~220) implies $u\in C^k(\mathbb{R})$ for each $k=1, 2,3,\cdots$.
 \end{proof}
 %
 \section{Conclusion} \label{sec:concl}
With the utilization of weak fractional R-L derivatives and appropriate fractional Sobolev space, we have established the existence and uniqueness of the strong solution to problem~\cref{eq:FODE}, together with its stability estimate and regularity.  The result suggests the suitability of utilizing fractional Sobolev spaces to analyze fractional R-L differential equations.
The whole framework laid out in this paper is applicable in a straightforward manner to ordinary differential equations that use Caputo fractional derivatives. This is mainly due to the strategy of using $C_0^\infty(\mathbb{R})$ for which Riemann-Liouville derivative coincides with Caputo derivative. We intend to adopt the main idea in the present paper to investigate fractional boundary value problems that can include non-constant coefficients.
\appendix
\renewcommand{\theequation}{\thesection.\arabic{equation}}
\numberwithin{equation}{section}
\section{Several Pertinent Theorems} \label{app:thm}
\begin{theorem}[Plancherel Theorem (see eg. \cite{MR924157} p. 187)] \label{thm:PAR}
Given $w\in L^2(\mathbb{R})$, there is a unique $\widehat{w}\in L^2(\mathbb{R})$ so that the following properties hold:
 \begin{itemize}
 \item If $w\in L^1(\mathbb{R})\cap L^2(\mathbb{R})$, then $\widehat{w} = \mathcal{F}(w)$.
 \item For every $w \in L^2(\mathbb{R})$, $\|w\|_2=\|\widehat{w}\|_2$.
 \item The mapping $w\rightarrow \widehat{w}$ is a Hilbert space isomorphism of $L^2(\mathbb{R})$ onto $L^2(\mathbb{R})$.
 \end{itemize}
\end{theorem}
\begin{theorem}[\cite{MR2597943}, p. 189]\label{thm:ParsevalFormula}
Given $v,w\in L^2(\mathbb{R})$, then $(v,\overline{w})=(\widehat{v},\overline{\widehat{w}}).$ and $w=(\widehat{w})^{\vee}.$
\end{theorem}
\begin{theorem}[\cite{MR1657104}, p. 204]\label{thm:ConvolutionFourierTransform}
If $v\in L^2(\mathbb{R}), w\in L^1(\mathbb{R})$, then $\widehat{v*w}=\widehat{v}\widehat{w} \in L^2(\mathbb{R})$.
\end{theorem}
\begin{theorem}[\cite{MR2328004}, Theorem 4.3-2, p. 191]\label{Thm:dense}
Let $(X,(\cdot,\cdot))$ be a Hilbert space and let $Y$ be a subspace of $X$. $\overline{Y}=X$ if and only if $0\in X$ is the only one satisfying $(x,y)=0$ for all $y\in Y$.
\end{theorem}
\begin{theorem}[\cite{MR2759829}, Proposition 4.20, p. 107]\label{thm:differentiability}
Let $v\in C_0^k(\mathbb{R})$ for $k\ge 1$ and $w\in L^1_{loc}(\mathbb{R})$. Then $v*w\in C^k(\mathbb{R})$ and $D^\alpha(v*w)=(D^\alpha v)*w$. In particular, if $v\in C_0^\infty(\mathbb{R})$, $w\in L^1_{loc}(\mathbb{R})$, then $v*w\in C^\infty(\mathbb{R})$.
\end{theorem}

\bibliographystyle{siamplain}
\bibliography{ftpbvp}
\end{nolinenumbers}
\end{document}